\documentclass[11pt,a4paper]{amsart}

\usepackage{lib/lib}
\usepackage{lib/style}
\usepackage{amssymb}
\usepackage{amsfonts}

\usepackage[colorinlistoftodos, textwidth=3cm, textsize=small]{todonotes}

\newcommand{\NN}{\mathbb{N}} 
\newcommand{\ZZ}{\mathbb{Z}}
\newcommand{\QQ}{\mathbb{Q}}
\newcommand{\RR}{\mathbb{R}}

\NewTextOperator{\conv}
\NewTextOperator{\vertices}[vert]
\NewTextOperator{\vol}
\NewTextOperator{\area}
\NewTextOperator{\sign}
\NewTextOperator{\slope}
\NewTextOperator{\affine}[aff]
\NewTextOperator{\cone}[cone][\left\langle][\right\rangle]

\DeclareMathOperator{\fra}{frac}

\DeclareMathOperator{\raycast}{raycast}
\DeclareMathOperator{\minsl}{minsl}
\DeclareMathOperator{\nextpt}{nextpt}
\DeclareDocumentCommand{\Minsl}{O{\Lambda}O{\ensuremath{H_n}}}{\ensuremath{\minsl_{#1,#2}}}
\DeclareDocumentCommand{\Nextpt}{O{\Lambda}O{\ensuremath{H_n}}}{\ensuremath{\nextpt_{#1,#2}}}
\newcommand{\NonZero}[1]{\ensuremath{{#1}{\scriptstyle\setminus\left\{{\bf 0}\right\}}}}

\newcommand{\floor}[1]{\left\lfloor #1 \right\rfloor}
\newcommand{\ceil}[1]{\left\lceil #1 \right\rceil}
\newcommand{\abs}[1]{\left\lvert #1 \right\rvert}

\DeclareDocumentCommand{\cardinal}{m}{{\left|{#1}\right|}}

\NewWordCommand{\wlg}[without loss of generality]
\NewWordCommand{\wrt}[with respect to]
\RenewWordCommand{\iff}[iff]

\DeclareDocumentCommand{\typeHint}{m}{\;{\color{lightgray} #1}}
\DeclareDocumentCommand{\textmono}{m}{{\normalfont \texttt{#1}}}

\begin{document}

\title{On the convex hull of integer points above the hyperbola%
}
    \thanks{%
        Supported by grants PID2022-137283NB-C21 and PRE2022-000020, %
        funded by MCIN/AEI/10.13039/501100011033.%
    }%

\author[D.~Alc\'antara]{David Alc\'antara}
\author[M.~Blanco]{M\'onica Blanco}
\author[F.~Criado]{Francisco Criado}
\author[F.~Santos]{Francisco Santos}

\address[D.~Alc\'antara, M.~Blanco, F.~Santos]{
    Departmento de Matem\'aticas, Estad\'istica y Computaci\'on\\
    Universidad de Cantabria\\
         Santander\\
         Spain}
\email{david.alcantara@unican.es,
monica.blancogomez@unican.es,\break
francisco.santos@unican.es}

\address[F.~Criado]{Departamento de Matem\'aticas \\
    CUNEF Universidad\\
        Madrid \\
        Spain}
\email{francisco.criado@cunef.edu}

\maketitle

\begin{abstract}
We show that the polyhedron defined as the convex hull of the lattice points above the hyperbola $\left\{xy = n\right\}$ has between $\Omega(n^{1/3})$ and $O(n^{1/3} \log n)$ vertices.
The same bounds apply to any hyperbola with rational slopes except that instead of $n$ we have $n/\Delta$ in the lower bound and by $\max\left\{\Delta, n/\Delta\right\}$ in the upper bound, where $\Delta \in \ZZ_{>0}$ is the discriminant.

We also give an algorithm  that enumerates the vertices of these convex hulls in logarithmic time per vertex.
One motivation for such an algorithm is the deterministic factorization of integers.
\end{abstract}

\setcounter{tocdepth}{2}
\tableofcontents

\section{Introduction}\label{sec:introduction} 

For any $n \in \ZZ_{>0}$, lattice points on the hyperbola
\[
    h_n \coloneqq \left\{(x, y) \in \RR^2 : xy = n\right\}
\]
are in correspondence with the factorizations of $n$ into two factors.
As the region of $\RR_+^2$ delimited by the inequality $xy \ge n$ is strictly convex,
all such points are also vertices of the convex hull of 
$H_n\cap \ZZ^2$ where
\begin{align}
\label{eq:H_n}
    H_n \coloneqq \left\{(x, y) \in \RR_{\ge 0}^2 :xy \ge n\right\},
\end{align}
denotes the region above (the positive branch of) the hyperbola.

With this motivation in mind, in this paper we ask ourselves how many vertices can
 $\conv{H_n\cap \ZZ^2}$ have, and how to enumerate them.
Observe that the number is finite, since they must all be contained in $\{1,\dots, n\}^2$.

We generalize this question to an arbitrary hyperbola $h \subset \RR^2$ with equation
\begin{align}
    \label{eq:hyp-intro}
    a (x-x_0)^2 + b (x-x_0)(y-y_0) + c (y-y_0)^2 = n,
\end{align}
with $a, b, c, n, x_0, y_0 \in \RR$.
The equation defines a hyperbola if and only if its discriminant $\Delta \coloneqq \sqrt{b^2 - 4ac}$ \ is
strictly positive and $n \ne 0$.
We no longer assume $n$ to be an integer.

Let us denote $H$ the closure of one of the convex components of $\RR^2 \setminus h$.
Then, $\conv{H \cap \ZZ^2}$ has finitely many vertices if and only if the asymptotes of $h$ have rational slopes. These slopes are
\begin{equation}
\label{eq:slopes}
    \frac{-b+\Delta}{2c} = \frac{2a}{-b-\Delta}
    \quad\text{ and }\quad
    \frac{-b-\Delta}{2c} = \frac{2a}{-b+\Delta}.
\end{equation}
Since multiplying Equation~\eqref{eq:hyp-intro} by a constant does not change the hyperbola, without loss of generality we assume $\Delta \in \QQ$.
Under this assumption the slopes are rational if $a, b$ and $ c$ are rational. The converse also holds: if the slopes and $\Delta$ are rational, then $a,b,c$ must be rational as well: for example, the difference of the slopes is $\frac{\Delta}{c}$. Similar arguments prove the rationality of $a$ and $b$.

Hence, we may also assume that $a,b,c$ and $\Delta$ are integer.
Our main result is:
\begin{theorem}
    \label{thm:main-new}
    There is a constant $C \in \RR$ such that the following holds.

    If $h$ is an arbitrary hyperbola with asymptotes of rational slope,
    given by equation~\eqref{eq:hyp-intro} with $a, b, c, \Delta \in \ZZ$,
    and $H$ is the closure of one
    of the convex components of $\RR^2 \setminus h$, then
    \[
        2\left(\frac{|n|}{\Delta}\right)^{1/3} - 2 \; \le  \; \cardinal{\vertices{\conv{H \cap \ZZ^2}}} \; \le\;  C m^{1/3} (\log_2{m} + 2),
    \]
    where $m \coloneqq \max\{2\Delta, \frac{|n|}{\Delta}\}$. 
\end{theorem}

One value of $C$ that works is
$C = \pi^{2/3} 2^{7/3} \simeq 10.810$, but a slight optimization gives the better value $C \simeq 8.205$.
See Theorem~\ref{thm:upper-bound-superlattice} and Remark~\ref{rmk:upper-bound-with-constant} for precise statements and more details.

Observe that assuming $|n| \ge 2\Delta^2$ Theorem~\ref{thm:main-new} gives
\[
    2m^{1/3} - 2 \le \cardinal{\vertices{\conv{H \cap \ZZ^2}}} \le C m^{1/3} (\log_2{m} + 2),
\]
where $m = |n|/\Delta$.
Hence, our bounds are tight modulo a logarithmic factor.

Specializing to the standard hyperbola, we have

\begin{corollary}
    \label{coro:main-new}
    There is a constant $C \in \RR$ such that
    \[
        2n^{1/3} - 2 \le \cardinal{\vertices{\conv{H_n\cap \ZZ^2}}} \le C n^{1/3} (\log_2{n} + 2).
    \]
\end{corollary}

These results fit in the framework of similar problems studied in convex geometry.
For example, in 1997 B\'ar\'any and
Larman~\cite{Barany1997ConvHullBall} proved that the number of
vertices of the convex hull of lattice points in the $(d - 1)$-sphere
of radius $\sqrt n$ is in
$\Theta\hspace{-0.25em}\left(\vphantom{n^2}\right.\hspace{-0.25em}%
    n^{\binom{d}{2} /(d + 1)}%
\left.\vphantom{n^2}\hspace{-0.25em}\right)$, generalizing the $2$-dimensional bound of
$\Theta\left(n^{1/3}\right)$
established before by Balog and
B\'ar\'any~\cite{BalogBarany1991ConvHullLatticeDisc}.
The upper bound follows from a much more general result of Andrews from 1963~\cite{Andrews1963ALB} bounding the number of vertices of a lattice polytope in terms of its volume.
Since we use that result to prove our upper bound, we state it as Lemma~\ref{lemma:andrews}.
See the survey by B\'ar\'any~\cite{Barany2008Survey} (2008) for more information on similar questions.

As a preparation for the proof of Theorem~\ref{thm:main-new}, in Section~\ref{subsec:reduction-standard-hyperbola} we reduce the case of an arbitrary hyperbola to a standard hyperbola, except now lattice points are considered with respect to a translated superlattice of the integers.
Once this is done, the upper and lower bounds of Theorem~\ref{thm:main-new} are proved in Sections~\ref{subsec:upper-bound} and~\ref{subsec:lower-bound} respectively. For the upper bound we cover the hyperbola with a logarithmic number of rectangles of controlled volume and apply to those rectangles the result of Andrews mentioned above. For the lower bound we show that the first $n^{1/3}$ lattice lines parallel to each of the two asymptotes of the hyperbola contain vertices of the convex hull.

We have computed $\cardinal{\conv{H_n\cap \ZZ^2}}$ for $n \in \left\{1, \dots, 1\, 000\,000\right\}$, and we plot the results against our upper and lower bounds in Section~\ref{subsec:experimental-results}. Although the $\Theta(n^{1/3})$ bound for the disc may suggest that our lower bound is close to tight while the upper bound may be prone to improvement, we would say that the experimental results are not conclusive.

\medskip

In the second part of the paper (Section~\ref{sec:algorithm}) we describe an algorithm to compute the convex hull with $O(\log n)$ operations per vertex:

\begin{theorem} \label{thm:main_algorithm}
    There is an algorithm that, with input $n \in \NN$ and a
    vertex ${\bf p}=({\bf p}_x, {\bf p}_y)$ of $H_n\cap \ZZ^2$, finds the next vertex
    ${\bf p}'=({\bf p}'_x, {\bf p}'_y)$ using at most $O(\log n)$ arithmetic operations.

    Here by ``next'' we mean that
    \[
        {\bf p}'_x = \min{\left\{x > {\bf p}_x : (x, y) \in \vertices{\conv{H_n\cap \ZZ^2}} \text{ for some } y\right\}}.
    \]
\end{theorem}

Hence:

\begin{corollary}
    \label{corollary:standard_hyperbola_convex_hull_alg}
    Let $n\in \NN$.
    There is an algorithm to compute $\vertices{H_n \cap \ZZ^2}$ with time
    complexity in
    \[
        O\big( N\log n\big) \le O\big( n^{1/3} \log^2 n\big),
    \]
    where $N=\cardinal{\vertices{\conv{H_n\cap \ZZ^2 }}}$.
\end{corollary}

Again, our  result is more general and can be applied to an arbitrary rational hyperbola. Our precise statement is:

\begin{theorem}
\label{thm:general_algorithm}
    Consider the hyperbola $h$ of Eq.~\eqref{eq:hyp-intro}, with $\Delta = \sqrt{b^2 - 4ac} > 0$ and
    $a, b, c, \Delta \in \ZZ$.
    There is an algorithm that with input $a,b,c,n$ and
    a vertex of $H\cap \ZZ^2$, finds the next vertex
    using at most 
    \[
        O(\log (\max\{\abs{n}, \abs{a}, \abs{b}, \abs{c}\}))
    \]
    arithmetic operations.
\end{theorem}

Our complexity bounds are in the \emph{arithmetic model}, or \emph{RAM model} of computation~\cite{Cormen2022IntroAlgo}. In this model we assume that comparisons, additions, substractions, multiplications, integer divisions, and integer square roots can all be computed in $O(1)$ time. Yet, our bound has a dependence on the bit size of the input (via the parameters $n$ and $\Delta$) which comes from the computation of $\gcd$'s.

\begin{remark}
\label{remark:complexity}

By integer divisions and square root we mean computing $\floor{b/a}$ and $\floor{\sqrt a}$ for input numbers $a,b\in \RR$. Including these as elementary operations in the model has a two-fold justification.
On the one hand, for word-sized integers they can be computed in a single machine instruction.
On the other hand, if multiplication of $L$-bit integers takes $M(L)$ bit operations, then we can also compute the inverse ($1/a$) and square root ($\sqrt a$) of an integer $a$ of that size in $O(M(L))$ operations via the Newton iterations $x_{k+1}=ax_k^2$ and $y_{k+1}=\tfrac{1}{2}\left(y_k+\tfrac{a}{y_k}\right)$ respectively;
while these iterations take $O(\log L)$ and $O(\log^2(L))$ products respectively, in order to compute $\floor{b/a}$ and $\floor{\sqrt a}$ we do not need the full precision of these products in each iteration. By duplicating the precision in each iteration, the division can be computed in $M(L)+M(\lceil L/2 \rceil)+M(\lceil L/4 \rceil) + \dots + 1$ bit operations which, assuming superlinearity of $M$, gives the stated $O(M(L))$ bit operations in total. 

In fact, this argument shows that all the complexity bounds in this paper translate to the bit model (assuming $x_0,y_0,n\in \QQ$) simply multiplying them by $M(L)$, where $L$ is the bit size of the input.
\end{remark}

To simplify the complexity analysis of our algorithm, in Section 3.1 we again first reduce the case of a general hyperbola to that of the standard hyperbola by changing the lattice, except now we assume the new lattice $\Lambda$ to be a sublattice of $\ZZ^2$ rather than a superlattice. We also introduce the concept of \emph{standard basis}, the computation of which for a given $\Lambda$ we consider a preprocessing.

Once this is done,
our algorithm is based on the observation that if ${\bf p}$ and ${\bf q}$ are consecutive vertices of $H_n \cap ({\bf p} + \Lambda)$ with ${\bf p}_x <{\bf q}_x$, then the slope of ${\bf q} - {\bf p}$ is the minimum slope among all the vectors in
\[
(H_n - {\bf p}) \cap \Lambda \cap \left\{(x, y)\in \RR^2: x \ge 0\right\} \subset \Lambda.
\]

Hence, we devise an algorithm to compute that minimum slope for any ${\bf p} \in H_n$ and any lattice $\Lambda$ (Algorithm~\ref{alg:compute_nextpt}). This algorithm maintains a lattice basis that is guaranteed to contain the slope we are seeking, iteratively narrowing the cone spanned by that basis in a way that parallels the steps taken by the Euclidean algorithm for computing a $\gcd$. See the statement of Theorem~\ref{thm:exists_find_nextpt}, and its proof in Section~\ref{subsec:rsmin}.

As said at the beginning of the introduction, the algorithm of Corollary~\ref{corollary:standard_hyperbola_convex_hull_alg}
gives, in particular, a deterministic
factorization algorithm for $n$ with running time $O(n^{1/3} \log^2 n)$.
Faster algorithms, running in time $O(n^{1/4 + \varepsilon})$, were obtained by Strassen~\cite{Strassen1976IntFactAlg},
and slightly improved by Bostan, Gaudry and Schost~\cite{BostanGaudrySchost2007LinRecIntFact}. Yet, our algorithm is quite flexible in that it can start
looking for integer vertices in the hyperbola (hence for factorizations of $n$) at any particular place, not necessarily at a vertex, at a cost of a logarithmic factor in the complexity:

\begin{proposition}[Corollary~\ref{cor:find_next_vertex_from_non_vertex_log}]
    \label{prop:intro_initial_vertex}
    There is an algorithm that, with input $n \in \NN$ and an $x_{start}\in [0,\infty)$, finds the next vertex
    ${\bf p}=({\bf p}_x, {\bf p}_y)\in H_n\cap \ZZ^2$ using at most $O(\log^2 n)$ arithmetic operations.

    Here by ``next'' we mean that
    \[
        {\bf p}_x = \min{\left\{x \ge x_{start} : (x, y) \in \vertices{H_n\cap\ZZ^2} \text{ for some } y\right\}}.
    \]
\end{proposition}

This is useful in several ways: for example, it allows for an easy parallelization of Corollary~\ref{corollary:standard_hyperbola_convex_hull_alg} by asking different processors to start searching at different values of $x_{start}$.
Observe also that the space complexity of Strassen's algorithm and its variants is of the same order $O(n^{1/4})$ as its time complexity, while ours has space complexity in $O(1)$ (in the arithmetic model, linear on the input size in the bit model).
In this respect, our algorithm is comparable to Lehman's~\cite{Lehman1974FactLargeInt}.

\section{The number of vertices}\label{sec:number-of-vertices}

\subsection{Reduction to the standard hyperbola}
\label{subsec:reduction-standard-hyperbola}
We here show that there is no loss of generality in assuming our
hyperbola to have asymptotes parallel to the coordinate axes, at the expense of changing
$\ZZ^2$ to a superlattice of it.

\begin{proposition}
    \label{prop:transform}
    Consider the hyperbola of Eq.~\eqref{eq:hyp-intro}, with $\Delta \coloneqq \sqrt{b^2 - 4ac} > 0$
    and $a, b, c, \Delta \in \ZZ$.
    Then, there is an affine transformation in $\RR^2$ sending it to the hyperbola
    $\{xy = \abs{n}/\Delta^2\}$ and sending $\ZZ^2$ to ${\bf p} + \Lambda$, where ${\bf p} \in \RR^2$ and $\Lambda$ is a
    superlattice of $\ZZ^2$ of index $\Delta$ (that is, of determinant $\Delta^{-1}$).
    Moreover, if $\gcd(a,b,c)=1$ then $(1, 0)$ and $(0, 1)$ are primitive vectors of $\Lambda$.
\end{proposition}
\begin{proof}
    Without loss of generality, we assume that $n > 0$.

    We treat first the case $ac = 0$.
    Assuming without loss of generality $c = 0$, the asymptotes of our
    hyperbola $h$ are parallel to the integer vectors
    \[
        {\bf v}_1 \coloneqq (b, -a) \qquad
        {\bf v}_2 \coloneqq (0,  1).
    \]
    Consider the linear transformation $f: \RR^2 \to \RR^2$ with matrix
    \[
        \frac 1 b
        \begin{pmatrix}
            1 & 0 \\
            a & b
        \end{pmatrix},
    \]
    which has $f({\bf v}_1) = {\bf e}_1$, $f({\bf v}_2) = {\bf e}_2$.
    It sends $h$ to the hyperbola
    \[
        \{(x - x_0')(y - y_0') = \tfrac{n}{b^2}\}
    \]
    where $(x_0', y_0') = f(x_0, y_0)$,
    and it sends the lattice $\ZZ^2$ to the superlattice $\Lambda \supseteq \ZZ^2$ generated by
    $f({\bf e}_1)=\tfrac1b(1, a)$ and $f({\bf e}_2)=(0, 1)$, with $\det \Lambda = b^{-1}$.
    If $\gcd(a, b, c) = 1$ and $c=0$, then $\gcd(a, b)=1$ and ${\bf v}_1$ and ${\bf v}_2$ are primitive vectors of $\ZZ^2$.
    Hence, their images, $(1, 0)$ and $(0, 1)$ are primitive as well in $\Lambda = f(\ZZ^2)$.
    If we now translate by ${\bf p} \coloneqq - (x_0', y_0')$, we have the stated result.

    We now look at the case $ac \ne 0$.
    Now the asymptotes of $h$ are parallel to
    \[
        {\bf v}_1 \coloneqq (-2c,   b + \Delta), \qquad
        {\bf v}_2 \coloneqq ( 2c, -(b - \Delta)),
    \]
    where $\Delta, b, c \in \ZZ$ by hypothesis.
    However, these vectors are not primitive
    since $b + \Delta$ and $b - \Delta$ must be both even: they have the same parity
    since their sum is even, and they cannot be both odd since $(b + \Delta)(b - \Delta) = 4ac$.

    Moreover, $\left(\tfrac{b + \Delta}2\right) \left(\tfrac{b - \Delta}2\right) = ac$ implies that we can write
    $c = c_1c_2$ and $a = a_1a_2$ with $a_1c_1 = \tfrac{b - \Delta}2$ and $a_2c_2 = \tfrac{b + \Delta}2$.
    This decomposition may not be unique, but one that works is to take $a_1=\gcd(a, \tfrac{b - \Delta}2)$ and then derive $a_2,c_1,c_2$ from that.
    Thus, we can reduce the asymptote vectors to the following, still integer, vectors:
    \[
        {\bf v}_1' \coloneqq \tfrac1{2c_2}{\bf v}_1 = \big(-c_1,  a_2\big), \qquad
        {\bf v}_2' \coloneqq \tfrac1{2c_1}{\bf v}_2 = \big( c_2, -a_1\big).
    \]
    The linear transformation $f$ which maps $\left\{{\bf v}_1', {\bf v}_2'\right\}$ to the canonical basis of
    $\ZZ^2$ is given 
    by the matrix
    \[
        \frac{1}{\Delta}
        \begin{pmatrix}
            a_1 & c_2 \\
            a_2 & c_1
        \end{pmatrix}.
    \]
    This transformation maps $\ZZ^2$ to a superlattice $\Lambda \supseteq \ZZ^2$, now with
    determinant $\det \Lambda = \tfrac{1}{\Delta}$, and it sends $h$ to the hyperbola
    \[
        \left\{(x - x_0')(y - y_0') = \tfrac{n}{\Delta^2}\right\},
    \]
    where again $(x_0', y_0') = f(x_0, y_0)$.
    Translating again by ${\bf p} \coloneqq -(x_0', y_0')$, we have the stated result. 

    To show primitivity in this case observe that $\gcd(c_1, a_2)$ divides $a$, $c$, $\tfrac{b + \Delta}2$
    and $\tfrac{b - \Delta}2$.
    Thus, it also divides $b$, so it divides $\gcd(a, b, c)$.
    The same can be said about $\gcd(c_2, a_1)$.
    Hence, $\gcd(a, b, c) = 1$ implies that ${\bf v}_1'$ and ${\bf v}_2'$ are primitive vectors in $\ZZ^2$, and so
    are their images $(1, 0)$ and $(0, 1)$ in $\Lambda$.
\end{proof}

Hence, proving properties of the integer convex hull of general hyperbolas
is equivalent to proving them for translated superlattice convex hulls of the standard
hyperbola.
With this in mind we introduce the following notation. With $H_n$ as in Eq.~\eqref{eq:H_n},
for each lattice $\Lambda \subset \RR^2$ and point ${\bf p} \in \RR^2$ we denote
\begin{equation}
\label{general_hyper}
    H_{n,{\bf p}+\Lambda} \coloneqq H_n \cap ({\bf p}+\Lambda).
\end{equation}

\subsection{Upper bound}\label{subsec:upper-bound}

In this section we prove the upper bound part of Theorem~\ref{thm:main-new} on the number of vertices of the hyperbola.
The affine transformation of Proposition~\ref{prop:transform} translates Theorem~\ref{thm:main-new} to the following,
more specific, result which includes an explicit value for the constant $C$.

\begin{theorem}\label{thm:upper-bound-superlattice}
    Let $n \in \RR_+$, ${\bf p} \in \RR^2$ and $\Lambda$ be a superlattice of the integer lattice,
    such that $(0, 1)$ and $(1, 0)$ are primitive vectors of $\Lambda$.
    Define
    $
    m \coloneqq 
        \tfrac1{\det \Lambda} \max\left\{n, 2\right\}.
    $

    Then, the number of vertices of $\conv{H_{n,{\bf p}+\Lambda}}$ is at most
    \[
    \left(\frac{2^8 \pi^2 (k-1)^2}{k}\right)^{\frac13} m^{\frac13} (\log_k{m} + 2) \in O(m^{1/3} \log m),
    \]
    for any real number $k \ge 2$.
\end{theorem}

Before proving this theorem let us show how it implies the upper bound of Theorem~\ref{thm:main-new}.

\begin{proof}[Proof of Theorem~\ref{thm:main-new}, upper bound.]
    Proposition~\ref{prop:transform} implies that counting the number of vertices in our original hyperbola $h$ of discriminant $\Delta$ is equivalent to counting them in the hyperbola $\{xy=n/\Delta^2\}$ with respect to a translation of a superlattice $\Lambda$ of determinant $\det{\Lambda} = 1/\Delta$.
    The assumption made in Theorem~\ref{thm:upper-bound-superlattice} that $(0, 1)$ and $(1, 0)$ are primitive in $\Lambda$ is, by Proposition~\ref{prop:transform}, equivalent to $\gcd(a, b, c) = 1$ in the original hyperbola $h$ of Theorem~\ref{thm:main-new}.
    This is no loss of generality since dividing the equation of $h$ by any common factor goes in our favour: it decreases $\Delta$ and it keeps $n/\Delta$ unchanged.
    Hence, we can apply Theorem~\ref{thm:upper-bound-superlattice} with
    \[
        m = \Delta \max\left\{n/\Delta^2, 2\right\} = \max\left\{n/\Delta, 2\Delta\right\}.
        \qedhere
    \]
\end{proof}

\begin{remark}[On the parameter $k$ and the constant $C$]
    \label{rmk:upper-bound-with-constant}
    The constant multiplying $m^{1/3}\log_2 m$ in Theorem~\ref{thm:upper-bound-superlattice} (hence the constant $C$ in Theorem~\ref{thm:main-new})
    is
    \[
        C \coloneqq \left(2^8 \pi^2\right)^{\frac13} \frac{(k-1)^{2/3}}{k^{1/3} \log_2 k} \simeq
        13.620 \, \frac{(k-1)^{2/3}}{k^{1/3} \log_2 k} = 13.620\, f(k),
    \]
    where we have separated in a function $f(k) \coloneqq \tfrac{(k-1)^{2/3}}{k^{1/3} \log_2 k}$ the factors that depend on the parameter $k \ge 2$ of the statement.
    This $k$, which we can choose, is the ratio of a geometric sequence used in the proof (see~\eqref{points_hyper} and Lemma~\ref{lemma:upper-bound:rectangles}).
    The most natural value for it would be $k = 2$, which gives
    \[
        f(2) = 2^{-\frac13} \simeq 0.794
        \quad\Rightarrow\quad
        C \simeq 13.620 \cdot 2^{-\frac13} = 10.810.
    \]
    But the minimum of $f(k)$ is obtained for $k \simeq 13.140$ and gives
    \[
        f(13.140) \simeq 0.602
        \quad\Rightarrow\quad
        C \simeq 13.620 \cdot f(13.140) \simeq 8.205.
    \]
\end{remark}
\medskip

The remainder of this section is aimed at proving Theorem~\ref{thm:upper-bound-superlattice}.
The general idea is to cover the set of lattice points that can potentially be vertices of $\conv{H_{n,{\bf p}+\Lambda}}$ by a logarithmic number of rectangles with controlled volume, and then apply the following general statement of Andrews~\cite{Andrews1963ALB} about the number of vertices of a lattice polytope with bounded volume:

\begin{lemma}[Andrews, 1963]\label{lemma:andrews}
    If $P \subseteq \RR^d$ is a polytope with integral vertices and nonempty
    interior, then:
    \begin{equation*}
        \cardinal{\vertices{P}} < \kappa_d\,\vol{P}^{\frac{d - 1}{d + 1}}.
    \end{equation*}

    Where $\kappa_d$ is a constant depending on the dimension. 
\end{lemma}

    The optimal constant for dimension $2$ in the above statement, which we denote
    $\kappa_2$, has been studied by several authors.
    B\'ar\'any and Tokushige~\cite{BaranyTokushige2004MinAreaConvNGons}
    show $\kappa_2 \ge \sqrt[3]{54} \simeq 3.780$ and write $\kappa_2$ as the minimum of a finite-dimensional optimization problem that they then try to solve.
    Their computations, although not complete, suggest that
    $\kappa_2$ is likely $\simeq 3.781$.
    Rabinowitz~\cite{Rabinowitz1993On3BoundAreaConvNGon} proves that the area of every
    lattice $n$-gon is at least
    $
        1/8\pi^2, 
    $
    which translates to
\begin{corollary}[Rabinowitz~\cite{Rabinowitz1993On3BoundAreaConvNGon}]
    \label{coro:Rabinowitz}
    If $P \subseteq \RR^2$ is a polygon with integral vertices and nonempty
    interior, then:
    \begin{equation*}
        \cardinal{\vertices{P}} < \kappa_2 \, \area{P}^{\frac13},
    \end{equation*}
    where $\kappa_2 = 2 \pi^{\frac23}$.
\end{corollary}

In order to apply this result we need to cover all vertices of $\conv{H_{n,{\bf p}+\Lambda}}$ by convex bodies.
A first, trivial, way of doing that is as follows:

\begin{lemma}
    \label{lem:square}
    All vertices of $\conv{H_{n,{\bf p}+\Lambda}}$ except for the first and last, are contained in the
    open square
    \[
        \left(\det \Lambda, \tfrac{n}{\det \Lambda}+1\right)^2.
    \]
\end{lemma}

\begin{proof}
    The leftmost vertical (lattice) line of $({\bf p}+\Lambda) \cap \left\{x > 0\right\}$ always contains a vertex.
    Furthermore, it always contains exactly one, as the vertical edge of
    $\conv{H_{n,{\bf p}+\Lambda}}$ is unbounded.
    The offset introduced by ${\bf p}$ may cause this vertical line to be
    arbitrarily close to the $y$ axis.
    However, all remaining vertices of $\conv{H_{n,{\bf p}+\Lambda}}$ must be in different
    vertical (lattice) lines of ${\bf p}+\Lambda$, all with $x$-coordinate strictly greater than $\det \Lambda$
    (distance between consecutive lattice lines),
    hence $y$-coordinate strictly less than $n/\det \Lambda +1$.
    The same reasoning for horizontal lines of $\Lambda$ allows us to deduce the statement.
\end{proof}

We are going to use a finer covering of the vertices of $\conv{H_{n,{\bf p}+\Lambda}}$.
For this, let $k \ge 2$ be a real parameter to be specified later (see Remark~\ref{rmk:upper-bound-with-constant}) and
consider the points in the hyperbola
\begin{equation}
\label{points_hyper}
    {\bf r}_i \coloneqq \left(k^i, \frac{n}{k^i}\right) \in h_n,
    \qquad i \in \ZZ,
\end{equation}
where we allow $i$ to be negative.

We define the axes-aligned
rectangles $(R_i)_{i \in \ZZ}$ having consecutive ${\bf r}_i$ as vertices; that is,
\[
    R_i \coloneqq \left[k^i, k^{i+1}\right] \times \left[\frac{n}{k^{i+1}}, \frac{n}{k^i}\right],  \qquad \area{R_i}=\frac{(k - 1)^2}{k} \frac{n}{\det \Lambda}
\]
Observe that every $R_i$ has the same area, equal to $\tfrac{(k - 1)^2}{k} \tfrac{n}{\det \Lambda}$ (where we consider area normalized to the lattice).
Additionally, we also consider the unit squares $U_i \coloneqq {\bf r}_i + [0, 1]^2$.
See Figure~\ref{fig:upper-bound-boxes} for the whole setup.

\begin{figure}[htb]
    \centering
    \includegraphics[width=.8\textwidth]{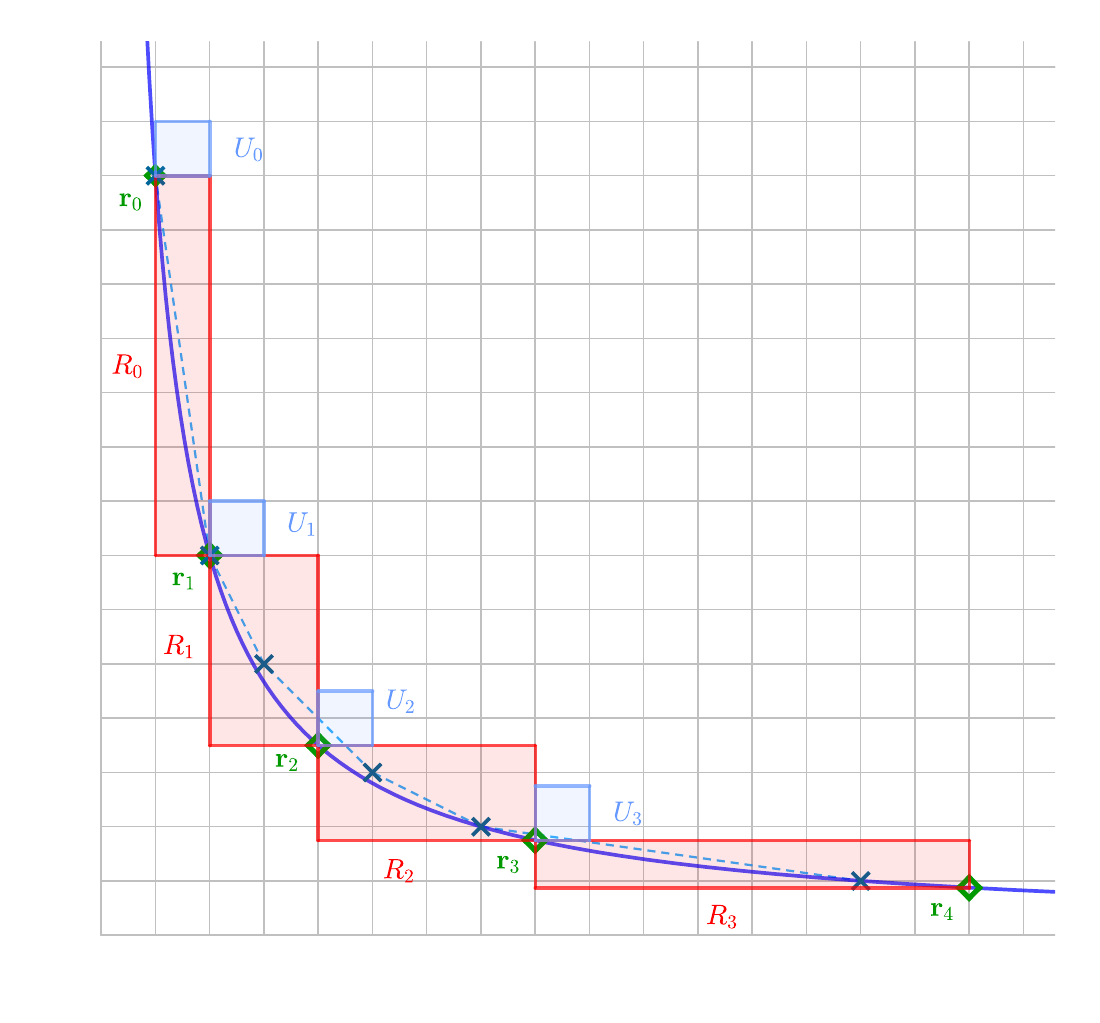}
    \caption{Rectangles $R_i$ and $U_i$ for $k=2$, $n = 14$ and $\Lambda = \ZZ^2$. By Lemma~\ref{lemma:upper-bound:rectangles} this gives $i \in \{0,1,2,3\}$.}
    \label{fig:upper-bound-boxes}
\end{figure}

The following lemma shows that to fence all vertices of $\conv{H_{n,{\bf p}+\Lambda}}$
we need only a logarithmic amount of these boxes.

\begin{lemma}\label{lemma:upper-bound:rectangles}
    Let ${\bf p} \in \RR^2$ and let  ${\bf q} \in \vertices{\conv{H_{n,{\bf p}+\Lambda}}}$ distinct from the leftmost or bottommost
    vertices, then there exists
    $i \in \ZZ$ with $\floor{\log_k{\det \Lambda}} \le i \le \floor{\log_k\left(\frac{n}{\det \Lambda}+1\right)}$
    such that ${\bf q} \in R_i$ or ${\bf q} \in U_i$.
\end{lemma}

\begin{proof}
    Consider ${\bf q} = ({\bf q}_x, {\bf q}_y) \in \vertices{\conv{H_{n,{\bf p}+\Lambda}}}$. 
    Let $i$ be such that ${\bf q}_x \in \left[k^i, k^{i+1}\right)$; that is, $i \coloneqq \floor{\log_k {\bf q}_x}$.
    By Lemma~\ref{lem:square}, $\det \Lambda < {\bf q}_x < \tfrac{n}{\det \Lambda}+1$, which gives
    $\floor{\log_k \det \Lambda} \le i \le \floor{\log_k\left(\tfrac{n}{\det \Lambda}+1\right)}$.
    
    By the definition of $i$, ${\bf q}$ is either in $R_i$ or directly above it; suppose the latter happens.
    Then, the vertical distance between ${\bf q}$ and $R_i$ must be smaller than $1$, since otherwise ${\bf q} - (0, 1)$ would be in $H_{n,{\bf p}+\Lambda}$,
    contradicting that ${\bf q}$ is a vertex.
    Once we know this, ${\bf q}$ must be in $U_i$ since otherwise ${\bf q} - (1, 0)$ would be in $H_{n,{\bf p}+\Lambda}$, again a contradiction with ${\bf q}$ being a vertex.
\end{proof}

We can now prove Theorem~\ref{thm:upper-bound-superlattice}, by adding up the areas of the $R_i$ and $U_i$ that contain the vertices of $H_{n,{\bf p}+\Lambda}$:

\begin{proof}[Proof of Theorem~\ref{thm:upper-bound-superlattice}]
    Recall that $\area{R_i}=\tfrac{(k - 1)^2}{k} \tfrac{n}{\det \Lambda}$, independent of $i$.
    We first look at the case when $\area{R_i}< 1$.
    Then we have
    \[
        \frac{n}{\det \Lambda} < \frac{k}{(k - 1)^2} \le 2,
    \]
    where the inequality in the right follows from $k \ge 2$.
    Since $\det{\Lambda} \le 1$, this implies $n \le 2$ so that $m = \tfrac2{\det \Lambda}\ge 2$.
    Lemma~\ref{lem:square} implies that all but at most $2$ vertices lie in $[0, 3]^2$, which is a square of (normalized) area $\tfrac9{\det \Lambda}$.
    Hence, by Corollary~\ref{coro:Rabinowitz} the total number of vertices is at most
    \[
        \frac{2 \pi^{2/3}9^{1/3}}{\det{\Lambda}^{1/3}} +2 =
        (2^2 9\pi^2m)^{1/3} +2<
        2 (2^8 \pi^2 m)^{1/3} <
        (2^8 \pi^2 m)^{1/3} (\log_k{m} + 2),
    \]
    where the inequality in the middle uses that $\max\{2^29\pi^2m,2^3\}\le 2^8\pi^2 m$ and the
    last inequality uses that $\log_k m \ge \log_k 2 \ge 0$.
    \medskip

    So, for the rest of the proof we assume that $\area{R_i} \ge 1$, giving
    \[
        1\le \area{R_i}= \frac{(k - 1)^2 n}{k \det{\Lambda}} \le \frac{(k - 1)^2}{2k} m.
    \]
    Although the rectangles $R_i$ and $U_i$ defined above are not
    necessarily lattice polygons in ${\bf p}+\Lambda$, the convex hull of the vertices
    they each contain, are.
    Furthermore, since the points in $\vertices{\conv{H_{n,{\bf p}+\Lambda}}}$ form an $x$-decreasing and
    $y$-increasing sequence, the vertices that lie in a particular $R_i$ are contained in the lower triangle of their
    bounding box; this bounding box is itself a rectangle contained in $R_i$, so
    \[
        \area{\conv{\vertices{H_{n,{\bf p}+\Lambda}} \cap R_i}}
        < \frac{1}{2}\area{R_i}
    \]

    Now, if the number of vertices lying in any region is $\ge 3$, their convex hull is $2$-dimensional, so unless there are $2$ vertices or less, we can apply Corollary~\ref{coro:Rabinowitz}.
    Thus we can conclude that the number of vertices within each $R_i$ is at most
    \[
        \max\left\{2, \kappa_2 \left(\frac{\area{R_i}}{2}\right)^{1/3}\right\} = \kappa_2 \left(\frac{\area{R_i}}{2}\right)^{1/3} \le
        \kappa_2 \left(\frac{(k - 1)^2}{2k} m\right)^{1/3}
    \]
    where the first equality uses that $\area{R_i}\ge 1$ and $\kappa_2 \coloneqq 2 \pi^{2/3} > 4$.

    Similarly, since $\area{U_i} = 1/\det{\Lambda} \ge 1$, the number of vertices in each $U_i$ is at most
    \begin{align*}
        \max\left\{2, \kappa_2 \left(\frac{\area{U_i}}{2}\right)^{1/3}\right\}
        &= \kappa_2 \left(\frac{\area{U_i}}{2}\right)^{1/3} =\\
        &= \kappa_2 \left(\frac1{2 \det{\Lambda}}\right)^{1/3} \le \\
        &\le \kappa_2 \left(\frac{2(k - 1)^2}{2k \det \Lambda}\right)^{1/3} \le
           \kappa_2 \left(\frac{(k - 1)^2}{2k} m\right)^{1/3}.
    \end{align*}
    We then have that the combined number of vertices in an $R_i$ and an $U_i$ is at most
    \[
        2 \kappa_2  \left(\frac{(k - 1)^2}{2k}m\right)^{\frac13} =
        \left(\frac{2^5 \pi^2 (k - 1)^2}{k} m\right)^{\frac13}.
    \]

    By Lemma~\ref{lemma:upper-bound:rectangles}, the number of boxes $R_i$ or $U_i$ that contain vertices other than the first and last is at most
    \begin{align*}
        \floor{{\log_k{\left(\frac{n}{\det \Lambda}+1\right)}}} + \ceil{\log_k{\frac{1}{\det \Lambda}}}
        &\le \log_k\left(m+1\right) + \log_k\frac{m}2 + 1 \le\\
        &\le \log_k\left(2m\right) + \log_k\frac{m}2 + 1=\\
        &=2 \log_k m + 2=
        2(\log_k{m} + 1),
    \end{align*}
    where we use that $m = \max\{\tfrac{n}{\det \Lambda}, \tfrac2{\det \Lambda}\}\ge 2$.

    Hence, the total number of vertices (including the first and last) is at most
    \[
        \left(\frac{2^8 \pi^2 (k - 1)^2}{k} m\right)^{\frac13} (\log_k{m} + 1) + 2 \le
        \left(\frac{2^8 \pi^2 (k - 1)^2}{k} m\right)^{\frac13} (\log_k{m} + 2).\qedhere
    \]
\end{proof}


\subsection{Lower Bound}\label{subsec:lower-bound}
In this section we prove the lower bound part of Theorem~\ref{thm:main-new}.
In fact, we prove the following stronger result about the distribution of the vertices.

\begin{theorem}\label{thm:closest-parallel-asymptotes-contain-vertex}
    Let $h$ be a general hyperbola with asymptotes of rational slope,
    as given by Eq.~\eqref{eq:hyp-intro} with $a, b, c, \Delta \in \ZZ$.
    Let $H$ be the closure of one of the convex regions of $\RR^2 \setminus h$.

    Then, the closest $\floor{\left(\tfrac{n}{\Delta}\right)^{1/3}}$ parallel lines to each
    asymptote of $h$ that intersect $H$ each contain a vertex of
    $\conv{H \cap \ZZ^2}$.
    Hence,
    \[
        \cardinal{\vertices{\conv{H}}} \ge 2\left(\frac{|n|}{\Delta}\right)^{1/3} - 2.
    \]
\end{theorem}

Once again, using Proposition~\ref{prop:transform}, this theorem follows immediately from the following equivalent result.

\begin{proposition}\label{prop:leftmost-vertical-lines-contain-vertex}
    Let $n \in \RR_+$, ${\bf p} \in \RR^2$ and let $\Lambda$ be a superlattice of the integer lattice.
    Then, the $\floor{\left(\frac{n}{\det \Lambda}\right)^{1/3}}$ leftmost vertical and bottommost horizontal lines of ${\bf p}+\Lambda$ that intersect
    $h_n$ each contain a vertex of $\conv{H_{n,{\bf p}+\Lambda}}$.

    Hence,
    \[
        \cardinal{\vertices{\conv{H_{n,{\bf p}+\Lambda}}}} \ge 2\left(\frac{|n|}{\det \Lambda}\right)^{1/3} - 2.
    \]
\end{proposition}

To prove this result, we use another transformation which allows
us to work under a more practical family of lattices.

We say that a lattice $\Gamma \subset \RR^2$ is a \emph{vertical shear} if it is generated
by the vectors $(1, s)$ and $(0, 1)$ for some $s \in \RR$.
Equivalently, if $\Gamma$ is unimodular and its vertical lattice lines are at
distance 1 to one another.

\begin{lemma}\label{lemma:point_in_unit_moon_vertex}
        Let $n \in \RR_+$, ${\bf p} \in \RR^2$ and $\Gamma$ a vertical
    shear lattice. If ${\bf q}= ({\bf q}_x, {\bf q}_y) \in H_{n,{\bf p}+\Lambda}$ is such that
    \[
        {\bf q}_y < \min\left\{\frac{n{\bf q}_x}{{\bf q}_x^2 - 1}, \frac{n}{{\bf q}_x} + 1\right\},
    \]
    then ${\bf q}$ is a vertex of $\conv{H_{n, {\bf p}+\Gamma}}$.
\end{lemma}
\begin{proof}

%

Let us consider the line $\ell$ with equation
    $y = n\frac{2{\bf q}_x - x}{{\bf q}_x^2 - 1}$ that cuts the hyperbola at $x\in\{{\bf q}_x-1, {\bf q}_x+1\}$.
     
    There is a vertex of $\conv{H_{n, {\bf p}+\Gamma}}$ in $({\bf q}_x - 1, {\bf q}_x + 1)\times \RR^2$ if and only if there are lattice points of $H_n$ in the vertical lattice line $\{x={\bf q}_x\}$ that are below $\ell$.

       The condition $ {\bf q}_y <\frac{n{\bf q}_x}{{\bf q}_x^2 - 1}$ guarantees ${\bf q}$ to be below $\ell$, and
       $ {\bf q}_y <\frac{n}{{\bf q}_x} + 1$ implies that ${\bf q} \in H_n$ is the bottommost such point on that vertical line, hence a vertex of $\conv{H_{n, {\bf p}+\Gamma}}$.
\end{proof}

%

\begin{lemma}\label{lemma:leftmost-cbrt-n-moon}
    Let $n \in \RR_+$, ${\bf p} \in \RR^2$ and let $\Gamma$ be a vertical shear lattice.
    Then, the leftmost $\floor{\sqrt[3]n}$ vertical lines of ${\bf p}+\Gamma$ each contain a vertex
    of $\conv{H_{n,{\bf p}+\Gamma}}$.
\end{lemma}

\begin{proof}
    Let ${\bf q} = ({\bf q}_x, {\bf q}_y) \in {\bf p}+\Gamma$ be the bottommost lattice point above the hyperbola along a lattice vertical line. In particular,  $ {\bf q}_y <\frac{n}{{\bf q}_x} + 1$.
    We have that
    \[
        {\bf q}_x \le n^{1/3} \Rightarrow {\bf q}_x({\bf q}_x^2 - 1) < n \Rightarrow (n + {\bf q}_x)({\bf q}_x^2 - 1) < n {\bf q}_x^2.
    \]
    If ${\bf q}_x \le 1$ then it is clear that ${\bf q}$ is a vertex (there is no other lattice vertical line with $0<x<{\bf q}_x$), and otherwise the above equation gives
    \begin{equation*}
        {\bf q}_y < \frac{n}{{\bf q}_x} + 1 < \frac{n {\bf q}_x}{{\bf q}_x^2 - 1},
    \end{equation*}
    so Lemma~\ref{lemma:point_in_unit_moon_vertex} implies that ${\bf q}$ is a vertex.
\end{proof}

\begin{proof}[Proof of Proposition~\ref{prop:leftmost-vertical-lines-contain-vertex}]
    Without loss of generality, assume that $(1, 0)$ and $(0, 1)$ are primitive in $\Lambda$.
    Consider the linear transformation $f$ which dilates the $x$ axis by a
    factor of $\tfrac1{\det \Lambda}$, 
    which maps $\Lambda$ to a vertical shear lattice $\Gamma$.
    It also maps the hyperbola $\left\{xy = n\right\}$ to a hyperbola $\left\{xy = n'\right\}$ with
    $n' \coloneqq \tfrac{n}{\det \Lambda}$, and it sends ${\bf p}$ to a certain point ${\bf p}'$.

    It follows from Lemma~\ref{lemma:leftmost-cbrt-n-moon} that the leftmost
    $\floor{\sqrt[3]{n'}}$ vertical lines of ${\bf p}'+\Gamma$ in the first quadrant,
    which are the image by $f$ of the leftmost $\floor{\sqrt[3]{n'}}$ vertical lines
    of ${\bf p}+\Lambda$, each contain a vertex of $\conv{H_{n',{\bf p}'+\Gamma}}$, image
    of $\conv{H_{n,{\bf p}+\Lambda}}$.
    With the same argument with a dilation in the $y$ axis, we conclude that
    the bottommost $\floor{\sqrt[3]{n'}}$ horizontal  lines of ${\bf p}+\Lambda$
    each contain a vertex of $\conv{H_{n,{\bf p}+\Lambda}}$.

    This gives a priori
    \[
        2\floor{\sqrt[3]{n'}} = 2\floor{\left(\frac{n}{\det \Lambda}\right)^{\frac13}} \ge 2\left(\frac{n}{\det \Lambda}\right)^{\frac13} - 2
    \]
    vertices, except we need to check that the ones obtained with horizontal and vertical lines are all different.
    If one vertex comes simultaneously from a vertical and horizontal line then $(\sqrt[3]{n'} \det \Lambda)^2 \ge n$, and we have that
    \[
        (\sqrt[3]{n'}\det \Lambda)^2 \ge n
        \quad\Leftrightarrow\quad
        n^{2/3}(\det \Lambda)^{4/3} \ge n
        \quad\Rightarrow\quad
        n \le (\det \Lambda)^4 \le \det \Lambda.
    \]
    Assuming that this does not happen is no loss of generality since $n \le \det \Lambda$ implies our lower bound to be zero.
\end{proof}
 
\subsection{Experimental Results}\label{subsec:experimental-results}

Let $V(n)$ denote the number of vertices of $\conv{H_{n,\ZZ^2}}$.

Using the algorithm of the next section, we have computed $V(n)$
for all values of $n$ up to $10^6$ and sparse values up to $10^{21}$.
The result is plotted in Figure~\ref{fig:vert_conv_H_n}.

\begin{figure}[H]
    \centering
    \includegraphics[width=.98\textwidth]%
        {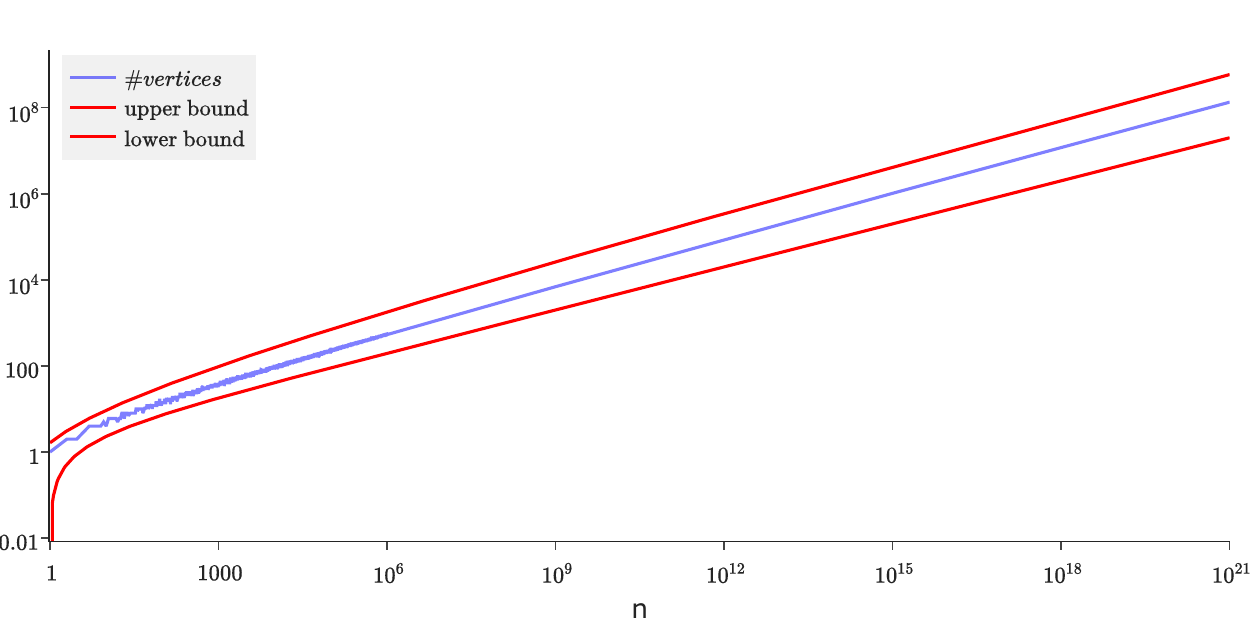}
    \caption{Number of vertices of $\conv{H_{n,\ZZ^2}}$, plotted against the upper bound of $8.2 n^{1/3}(\log_2{n} + 2)$ and the lower bound of $2n^{1/3} - 2$. Plot is logarithmic in both axes.}
    \label{fig:vert_conv_H_n}
\end{figure}

In order to have experimental evidence of whether $V(n)$ behaves closer to our upper bound or to our lower bound,
Figures~\ref{fig:vert_conv_H_n__div__cbrt_n_ln_n} and~\ref{fig:vert_conv_H_n__div__cbrt_n} show
respectively $V(n)/n^{1/3}\log n$ and $V(n)/n^{1/3}$.
In lighter blue we plot the individual values for each $n$, while in darker blue we plot the moving averages of 2000 and 20000 values of $n$ respectively.
It is difficult to draw conclusions but, to the naked eye, the first of these two plots seems closer to a constant for high values of $n$, suggesting that the true asymptotics of $V(n)$ is closer to $n^{1/3}\log n$ than to $n^{1/3}$.

\begin{figure}[H]
    \centering
    \includegraphics[width=.98\textwidth]%
        {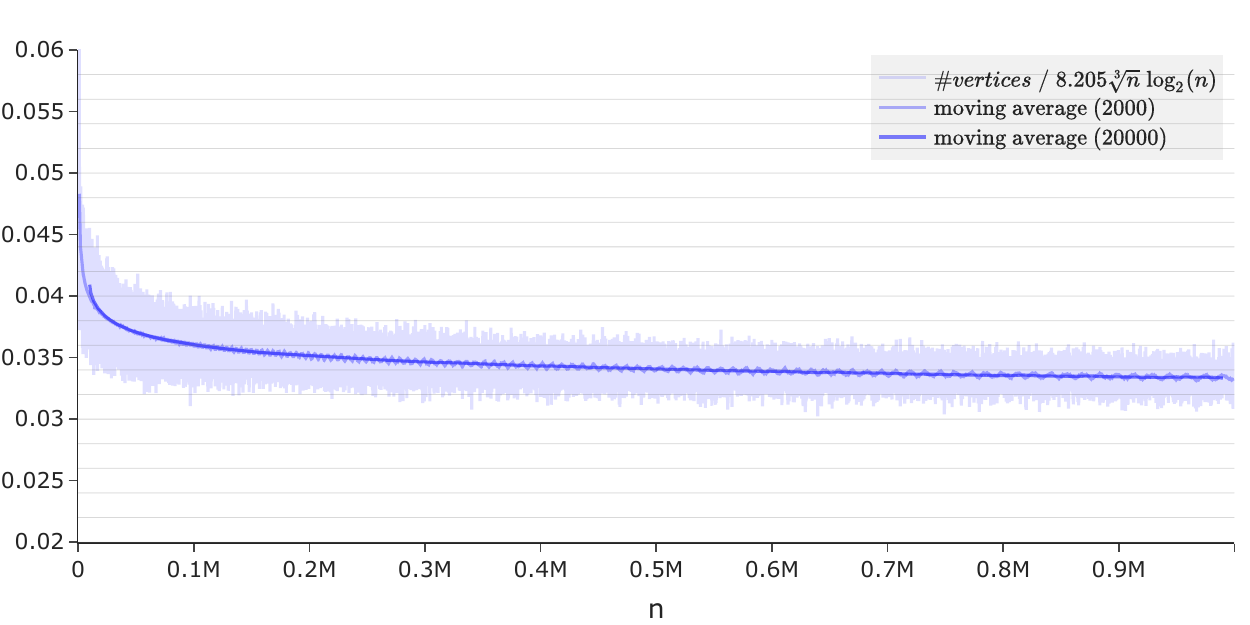}
    \caption{Number of vertices of $\conv{H_{n,\ZZ^2}}$ divided by $8.2n^{1/3} \log_2{n}$.}
    \label{fig:vert_conv_H_n__div__cbrt_n_ln_n}
\end{figure}

\begin{figure}[H]
    \centering
    \includegraphics[width=.98\textwidth]%
        {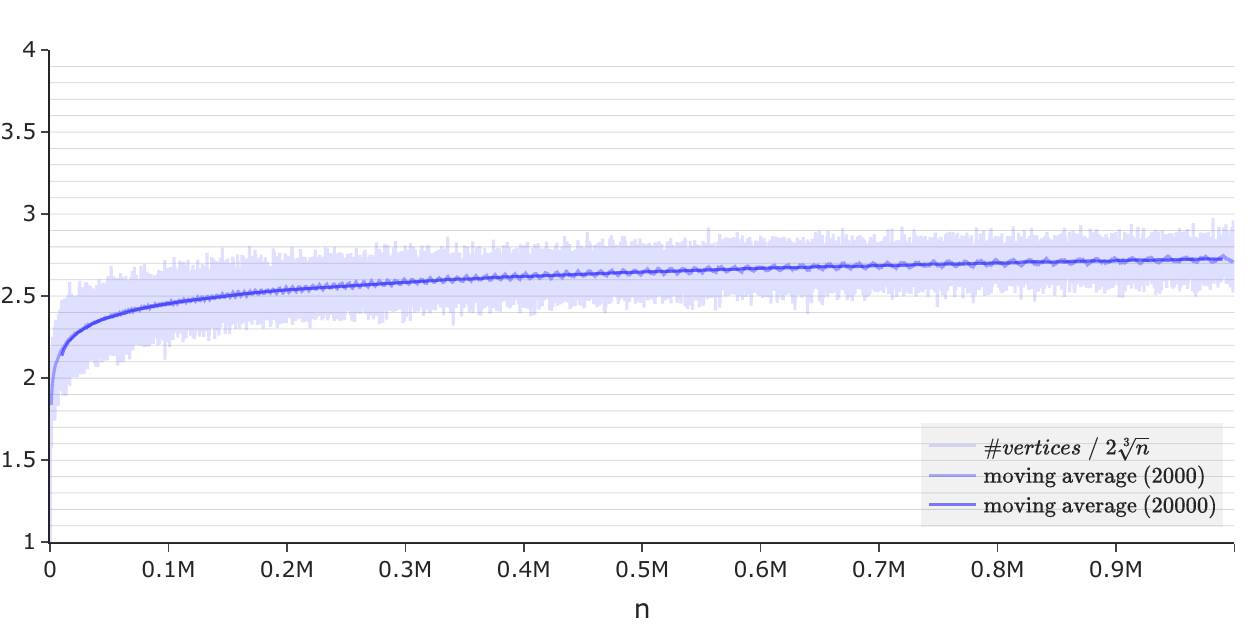}
    \caption{Number of vertices of $\conv{H_{n,\ZZ^2}}$ divided by $2n^{1/3}$.}
    \label{fig:vert_conv_H_n__div__cbrt_n}
\end{figure}

\section{Algorithm to enumerate the vertices}\label{sec:conv-hull-algorithm}
\label{sec:algorithm}

In this section we propose an algorithm that is able to find the vertices of $\conv{H_{n,\ZZ^2}}$
iteratively in $O(\log n)$ time per vertex.
Our algorithm is also applicable to find the vertices of $\conv{H_{n,{\bf p}+\Lambda}}$
for an arbitrary ${\bf p}\in H_n$ and $\Lambda$ an arbitrary rational lattice. However, whenever we state  complexity bounds we will typically assume $\Lambda$ to be a
reduced  sub-lattice of $\ZZ^2$, with the following definition:

\begin{definition}
\label{def:reduced_sublattice}
    A sublattice $\Lambda \subset \ZZ^2$ is called \emph{reduced} if it satisfies the following (equivalent) conditions:
    \begin{enumerate}
        \item $\pi_1(\Lambda)=\pi_2(\Lambda) = \ZZ$, where $\pi_i: \RR^2 \to \RR$ are the coordinate projections.
        \item $(\det \Lambda, 0)$ and $(0, \det \Lambda)$ are primitive lattice vectors.
        \item For some basis (equivalently, for any basis) $\{{\bf b}_1, {\bf b}_2\}$ of $\Lambda$ we have that
        \[
            \gcd({{\bf b}_1}_x, {{\bf b}_2}_x) = \gcd({{\bf b}_1}_y, {{\bf b}_2}_y) = 1.
        \]
    \end{enumerate}
\end{definition}

Observe that part (3) provides an easy way to transform any rational lattice into a reduced one: simply divide the $x$ and $y$ coordinates by the $\gcd$'s in the statement.

\subsection{Preprocessing: Reduction to the standard hyperbola and computation of standard basis}
\label{subsec:preprocessing}

The following version of Proposition~\ref{prop:transform} tells us how to translate results and algorithms on the standard hyperbola with respect to a reduced sublattice to an arbitrary hyperbola with asymptotes of rational slope.

\begin{lemma}
    \label{lemma:transform_sublattice}
    Consider the hyperbola $h$ of Eq.~\eqref{eq:hyp-intro}, with $\Delta = \sqrt{b^2 - 4ac} > 0$,
    $a, b, c, \Delta \in \ZZ$ and $x_0,\;y_0 \in \RR$.
    Then, there is an affine transformation in $\RR^2$ sending $h$ to the hyperbola
    $\{xy = \abs{n}\}$ and sending $\ZZ^2$ to a sublattice ${\bf p}+\Lambda \subseteq \ZZ^2$,
    with ${\bf p} \in \RR^2$ and $\det{\Lambda} = \Delta$.
    Moreover:
    \begin{enumerate}
        \item If $\gcd(a,b,c)=1$ then $\Lambda$ is reduced.
        \item     In $O(\log \max\{\abs{a}, \abs{b}, \abs{c}\})$ operations we can compute ${\bf p}$ and a basis of $\Lambda$ with coordinates bounded by $\max\{\abs{a}, \abs{b}, \abs{c}\}$.

    \end{enumerate}
\end{lemma}

\begin{proof}
    Via Proposition~\ref{prop:transform} we first transform the hyperbola into
    $\left\{xy = \abs{n}/\Delta^2\right\}$, with an affine map
    that sends $\ZZ^2$ to an affine lattice ${\bf q}+\Lambda'$ where ${\bf q} \in \RR^2$ and $\Lambda' \supseteq \ZZ^2$ is a superlattice of
    determinant $\Delta^{-1}$.
    By the calculations in the proof of that proposition, the point ${\bf q}$ is obtained as
%
    \[
        {\bf q} = - f(x_0,y_0)=-\frac{1}{\Delta}
        \begin{pmatrix}
            a_1 & c_2 \\
            a_2 & c_1
        \end{pmatrix}
        \begin{pmatrix}
            x_0 \\
            y_0
        \end{pmatrix},
    \]
    where $f$ is the linear map associated with the affine map, and $a_1,a_2, c_1,c_2$ are integers bounded in absolute value by $\max\{\abs{a}, \abs{b}, \abs{c}\}$ (this holds even in the case $ac=0$).
    Thus, scaling both coordinates by $\Delta $ the hyperbola is further transformed to the
    standard hyperbola $\left\{xy = \abs{n}\right\}$, and the lattice to ${\bf p} +\Lambda$,
    where $\Lambda\coloneqq \Delta \,\Lambda'$ is a sublattice of $\ZZ^2$ of determinant $\Delta $ and ${\bf p}\coloneqq \Delta{\bf q}$.

    The first property after the ``moreover'' follows from the fact that if $\gcd(a,b,c)=1$ then $(1,0)$ and $(0,1)$ are primitive in $\Lambda'$, hence $(\Delta,0)$ and $(0,\Delta)$ are primitive in $\Lambda$.

    For the second property, we multiply by $\Delta$ the basis of $\Lambda'$ obtained naturally in Proposition~\ref{prop:transform}, that is, the image of ${\bf e}_1, {\bf e}_2$ by the linear transformation $f$. More precisely, the basis is:
    \[
        {\bf b}_1= \Delta f({\bf e}_1)=  
             \Delta \begin{pmatrix}
            a_1 \\
            a_2
        \end{pmatrix},
        \quad
        {\bf b}_2 =    \Delta f({\bf e}_1)=
             \Delta \begin{pmatrix}
            c_2 \\
            c_1
        \end{pmatrix}.
    \]
    As seen in Proposition~\ref{prop:transform}, to compute the numbers $a_1,a_2, c_1, c_2$ we can take $a_1=\gcd(a, \tfrac{b - \Delta}2)$, which is computed in $O(\log \max\{\abs{a}, \abs{b}, \abs{c}\})$, and compute the rest from $a_1$ in $O(1)$ time.
\end{proof}

Our algorithms use a specific basis of any lattice $\Lambda\subset \QQ^2$ that we call standard:
\begin{definition}
    \label{def:standard_basis}
    Let $\Lambda \subset \QQ^2$ be a rational lattice.
    The \emph{standard basis} of $\Lambda$ is the unique basis ${\bf b}_1 = ({{\bf b}_1}_x, {{\bf b}_1}_y)$, ${\bf b}_2 = (0, {{\bf b}_2}_y)$ satisfying
    \begin{equation}
        \label{eq:basis}
        {{\bf b}_1}_x > 0, \quad {{\bf b}_2}_y > 0, \quad 0 \le {{\bf b}_1}_y < {{\bf b}_2}_y.
    \end{equation}
\end{definition}

The standard basis of a  reduced lattice $\Lambda$ has the form ${\bf b}_1 = (1, {{\bf b}_1}_y)$, ${\bf b}_2 = (0, \det \Lambda)$, with ${{\bf b}_1}_y \in \{0,\dots, \det \Lambda -1\}$. More generally:

\begin{proposition}
    \label{prop:standard_basis}
    The standard basis exists and is unique for every rational lattice $\Lambda\subset \QQ^2$.
    It can be computed from an arbitrary basis of $\Lambda$ in time linear on the bit size of that basis. If $\Lambda\subset \ZZ^2$, its coordinates are bounded by $\det(\Lambda)$.
\end{proposition}

\begin{proof}
    Existence and uniqueness follows from the following description:
    \begin{align*}
        {{\bf b}_1}_x &= \min\{x > 0: (x, y) \in \Lambda \text{ for some $y$}\}, \\
        {{\bf b}_2}_y &= \min\{y > 0: (0, y) \in \Lambda\}, \\
        {{\bf b}_1}_y &= {\bf b}_y \bmod {{\bf b}_2}_y,
    \end{align*}
    where ${\bf b}_y \in \QQ$ is any number such that $({{\bf b}_1}_x, {\bf b}_y) \in \Lambda$ and, for two rational numbers $p, q$ we denote
    \[
        p \bmod q \coloneqq \ceil{p/q}q - p.
    \]

    These three numbers can be computed from an arbitrary basis ${\bf w}_1, {\bf w}_2$ of $\Lambda$ as follows:
    \begin{align*}
        {{\bf b}_1}_x &= \gcd({{\bf w}_1}_x, {{\bf w}_2}_x), \\
        {{\bf b}_2}_y &= \det(\Lambda)/{{\bf b}_1}_x,\\
        {{\bf b}_1}_y &= (s\cdot{{\bf w}_1}_y + t\cdot{{\bf w}_2}_y) \bmod {{\bf b}_2}_y,\quad \text{ where }{{\bf b}_1}_x = s\cdot{{\bf w}_1}_x + t\cdot{{\bf w}_2}_x.
    \end{align*}
    (That is, $s$ and $t$ are the Bezout coefficients obtained from the extended Euclidean algorithm in the computation of $\gcd({{\bf w}_1}_x, {{\bf w}_2}_x)$).

    The complexity of this computation is dominated by the $\gcd$; hence, in the arithmetic model, it is linear in the bit size of the input basis. The size of the coordinates follows from the fact that $\det(\Lambda) = {{\bf b}_1}_x {{\bf b}_2}_y$ and ${{\bf b}_1}_y < {{\bf b}_2}_y$.
\end{proof}

\subsection{The algorithms}\label{subsec:putting-things-together}

Via the transformations from the previous section, in this section we assume without loss of generality that we are given an $n>0$, ${\bf p} \in H_n$ and the standard basis $({\bf b}_1, {\bf b}_2)$ of a rational lattice $\Lambda\subset \ZZ^2$.
We want to compute the vertices of $\conv{H_{n,{\bf p}+\Lambda}}$.
In some of the statements (typically those related to complexity bounds) we need to assume that $\Lambda$ is a reduced sublattice of $\ZZ^2$. 

\subsubsection{Finding the first vertex}

The following result tells us how to find the vertex with the smallest $x$-coordinate.

\begin{proposition}
\label{prop:first_vertex}
Let ${\bf p} \in H_n$ and let $\Lambda\subset \QQ^2$ be a rational lattice.
Having a standard basis of $\Lambda$, the first vertex of $\conv{H_{n,{\bf p} + \Lambda}}$ can be computed in $O(1)$ arithmetic operations.
\end{proposition}

\begin{proof}
    The first vertical line of ${\bf p}+\Lambda$ intersecting the positive open quadrant has
    $x$-coordinate congruent to ${\bf p}_x$ modulo ${{\bf b}_1}_x$. Hence, one of its lattice points is
    \[
        {\bf q} \coloneqq {\bf p} - \floor{\frac{{\bf p}_x - 1}{{{\bf b}_1}_x}} {\bf b}_1.
    \]
    The first vertex of $\conv{H_n,{\bf p} + \Lambda}$ is just the bottommost point of $H_{n,{\bf p}+\Lambda}$ along that line, which equals
    \[
        {\bf q} + \ceil{\frac{n/{\bf q}_x - {\bf q}_y}{{{\bf b}_2}_y}} {\bf b}_2. \qedhere
    \]
\end{proof}

\subsubsection{Finding the next vertex}

To find the next vertex of a given vertex we introduce the following concepts.

\begin{definition}
    Let ${\bf p} \in  H_n$ and let $\Lambda \subseteq \QQ^2$ be a
    rational lattice.
    We denote by $\Minsl({\bf p})$ the primitive non-zero vector
    with the smallest slope in
    \[
        (H_{n,{\bf p}+\Lambda} - {\bf p}) \cap \left\{(x, y): x \ge 0\right\} \subset \Lambda,
    \]
    and by
    $\Nextpt({\bf p})$ the furthest point of $H_{n,{\bf p}+\Lambda}$ along the ray from ${\bf p}$ in that direction; that is,
    \begin{align*}
        &\Nextpt({\bf p}) \coloneqq
            {\bf p} + k\cdot\Minsl({\bf p}),
            \qquad \text{ where}\\
        &\,k = \max{\{j \in \ZZ: {\bf p} + j\cdot\Minsl({\bf p}) \in H_n\}} \ge 1,
    \end{align*}
    If $ k = \infty$ we take $\Nextpt({\bf p}) ={\bf \infty}$.
\end{definition}

\begin{remark}
    Let ${\bf b}_1 = ({{\bf b}_1}_x, {{\bf b}_1}_y), {\bf b}_2 = (0, {{\bf b}_2}_y)$ be the standard basis of a rational lattice $\Lambda$, and  ${\bf p}= ({\bf p}_x, {\bf p}_y) \in H_n$.
    \begin{itemize}
        \item If ${\bf p} - {\bf b}_2 \in H_n$ (equivalently, if ${\bf p}_y \ge n/{\bf p}_x + {{\bf b}_2}_y$) then $\Minsl({\bf p}) = -{\bf b}_2$ and
        $\Nextpt({\bf p})$ is the bottommost lattice point of $H_{n,{\bf p}+\Lambda}$ in the vertical line containing ${\bf p}$.
        Put differently,
        \[
            \Nextpt({\bf p}) = {\bf p} + \ceil{\frac{n/{\bf p}_x- {\bf p}_y }{{{\bf b}_2}_y}} {\bf b}_2.
        \]
        Moreover, $\Nextpt({\bf p})$ is the first vertex of $\conv{H_{n,{\bf p}+\Lambda} \cap\{x \ge {\bf p}_x\}}$.

        \item If ${\bf p} - {\bf b}_2 \notin H_n$ (equivalently, if $n/{\bf p}_x \le {\bf p}_y < n/{\bf p}_x + {{\bf b}_2}_y$) then ${\bf p}$ is the first vertex of
        $
            \conv{H_{n,{\bf p}+\Lambda} \cap\{x \ge {\bf p}_x\}}
        $
        and
        $\Nextpt({\bf p})$ is the next one (or $\infty$, if there is no vertex after ${\bf p}$).
    \end{itemize}
\end{remark}

Hence:

\begin{corollary}
\label{cor:nextpt_is_next_vertex}
    If ${\bf p}$ is a vertex of $\conv{H_{n,{\bf p}+\Lambda}}$ then $\Nextpt({\bf p})$ is the next one.
    \qed
\end{corollary}

To analyze the complexity of our main algorithm we need the following:

\begin{lemma}
\label{lemma:exception}
    Let $\Lambda\subset \ZZ^2$ be a reduced lattice and let
     ${\bf p} \in H_n$.
     \begin{enumerate}
     \item If ${\bf p}_y\in (1,2)$ then $\Nextpt({\bf p})$ can be computed from the standard basis of $\Lambda$ in $O(\log \det \Lambda)$ operations.
     \item If ${\bf p}_y\notin (1,2)$ then the $x$ coordinate of $\Minsl({\bf p})$ is bounded above by $2n+2\det(\Lambda)$.
     \end{enumerate}
\end{lemma}

\begin{proof}
    We assume $\Nextpt({\bf p})\ne \infty$ (equivalently, ${\bf p}_y > 1$), since otherwise $\Minsl({\bf p}) = (\det(\Lambda), 0)$. Let ${\bf q} = \Nextpt({\bf p}) $.
    Observe that ${\bf q}$ is the first lattice point in the horizontal half-line $\{y={\bf q}_y\} \cap \{x \ge \max({\bf p}_x, n/ {\bf q}_y)\}$.
    In particular, ${\bf q}_x$ is bounded above by $\max\{{\bf p}_x, n/ {\bf q}_y\} + \det \Lambda$ and the $x$-coordinate of $\Minsl({\bf p})$ is bounded by
    \begin{align}
    \label{eq:nextpt_x}
    {\bf q}_x - {\bf p}_x \le \max\{0, n/ {\bf q}_y-{\bf p}_x\} + \det \Lambda.
    \end{align}

    If $1 < {\bf p}_y < 2$ we clearly have ${\bf q}_y = {\bf p}_y - 1$ and the observation above allows us to compute ${\bf q}_x$ as follows:
    Let ${\bf q}'$ be the last vertex of $\conv{H_{n,{\bf p} +\Lambda}}$, which is the point in ${\bf p} + \Lambda$ with ${\bf q}'_y = {\bf p}_y -1$ with the minimum $x$-coordinate beyond $n/{\bf q}'_y$.
    Take:
    \begin{align*}
    \label{eq:last_vertex}
    {\bf q} = {\bf q}' + \ceil{ \frac{ \max({\bf p}_x - {\bf q}'_x, 0)}{\det \Lambda} }  (\det \Lambda,0).
    \end{align*}
    To find ${\bf q}'$ we first need to compute the standard basis of the reflected lattice $\Lambda'\coloneqq\{(y,x): (x,y)\in \Lambda\}$, which takes $O(\log \det \Lambda)$ operations.
    After that, we can apply Proposition~\ref{prop:first_vertex} with the $x$ and $y$ coordinates swapped, which takes $O(1)$ operations more.

    For the second part of the statement we assume ${\bf q}_x - {\bf p}_x > 2n+2\det \Lambda$. Let us denote $\fra( {\bf p}_y) \coloneqq  {\bf p}_y - \floor{{\bf p}_y}$ and let ${\bf p}' \in {\bf p} + \Lambda$ be the first lattice point with ${\bf p}'_y=  \fra({\bf p}_y)+1$ and
    $
    {\bf p}'_x \ge \max \{ {\bf p}_x, n/{\bf p}_y'\}.
    $
    Our claim is that in fact ${\bf p}' = {\bf p}$, which implies the statement since then
    \[
    {\bf p}_y =  {\bf p}'_y  = \fra({\bf p}_y)+1 <2.
    \]

    We prove the claim by contradiction, so suppose ${\bf p}' \ne {\bf p}$.
    By construction of ${\bf p}' $ we have that ${\bf p}' \in H_{n, {\bf p} + \Lambda} \cap \{x \ge {\bf p}_x\}$, so the slope of ${\bf q}- {\bf p}$ should be less than (or equal to) the slope of ${\bf p}'- {\bf p}$.
    Equivalently, since ${\bf q}_y < {\bf p}'_y < {\bf p}_y$, the slope of ${\bf q} - {\bf p}'$ is less than (or equal to) the slope of ${\bf p}' - {\bf p}$.
    The same argument that proved Eq.~\eqref{eq:nextpt_x}, applied to ${\bf p}'$ instead of ${\bf q}$, implies that 
    \[
        {\bf p}'_x - {\bf p}_x \le \det \Lambda + \max\{0, n/ {\bf p}'_y-{\bf p}_x\} \le  n + \det \Lambda,
    \]
    and from this we get
    \[
        {\bf q}_x - {\bf p}'_x  = ({\bf q}_x - {\bf p}_x) - ({\bf p}'_x - {\bf p}_x) > (2n+2\det\Lambda) - (n + \det \Lambda) = n + \det \Lambda.
    \]
    On the other hand, assuming ${\bf p}' \ne {\bf p}$ we have
    \begin{align*}
    {\bf p}_y - {\bf p}'_y \ge 1,
    \qquad \text{ and } \qquad
    {\bf p}'_y - {\bf q}_y =  1.
    \end{align*}
    These four equations imply
    \[
    \slope{{\bf q}- {\bf p}'} > \frac1{n+\det\Lambda} \ge \slope{{\bf p}'- {\bf p}},
    \]
    a contradiction.
\end{proof}

We can now state our main algorithmic result, computing $\Nextpt({\bf p})$ in logarithmic time.
Its (long) proof is deferred to Section~\ref{subsec:rsmin}.

\begin{theorem}
    \label{thm:exists_find_nextpt}
    \label{thm:exists_find_next_vertex}
    Let $\Lambda \subset \QQ^2$ be a rational lattice.
    Given an $n  > 0$, a point ${\bf p} \in  H_n$, and the standard basis  ${\bf b}_1, {\bf b}_2$ of $\Lambda$, Algorithm~\ref{alg:compute_nextpt} (described in the next section) computes $\textmono{nextpt}({\bf b}_1, {\bf b}_2, n, {\bf p})\coloneqq\Nextpt({\bf p})$.

    If $\Lambda$ is a reduced sublattice of $\ZZ^2$, the number of operations taken by the algorithm is in $O(\log (n+\det \Lambda))$  except (perhaps) if ${\bf p}_y \in (1,2)$.
\end{theorem}

We can now prove Theorem~\ref{thm:general_algorithm}.

\begin{proof}[Proof of Theorem~\ref{thm:general_algorithm}]
We first apply Lemma~\ref{lemma:transform_sublattice} to find an affine transformation sending our original hyperbola $h$ to the standard hyperbola $\{xy = \abs{n}\}$, the starting vertex to a point ${\bf p}\coloneq ({\bf p}_y,{\bf p}_x) \in \RR^2$ and the integer lattice to a reduced lattice $\Lambda\subset \QQ^2$, of determinant $\Delta$. This step takes $O(\log (\max\{\abs{a}, \abs{b}, \abs{c}\}))$ operations and provides a basis of $\Lambda$ with coordinates bounded by $\max\{\abs{a}, \abs{b}, \abs{c} \}$, so the standard basis of $\Lambda$ can also be computed in the same number of operations.

We then apply Theorem~\ref{thm:exists_find_nextpt}, which allows us to compute the next vertex in time $O(\log(\abs{n}+\Delta))$. This is in $O(\log (\max\{\abs{n}, \abs{a}, \abs{b}, \abs{c}\}))$ since
\[
\abs{n}+\Delta  \le \abs{n} + 3 \max\{\abs{a}, \abs{b}, \abs{c}\} \le 4 \max\{\abs{n}, \abs{a}, \abs{b}, \abs{c}\}.
\]

We have to deal with the exceptional case of Theorem~\ref{thm:exists_find_nextpt}, namely the case ${\bf p}_y \in (1,2)$. If this happens we apply Lemma~\ref{lemma:exception} instead, to find the next (and last) vertex in the slightly better number $O(\log (\Delta)) \le O(\log(\max\{\abs{a}, \abs{b}, \abs{c}\}))$ of operations .
\end{proof}

\subsubsection{Computing all vertices}

Let us now show  how to compute the convex hull
of $H_{n,{\bf p}+\Lambda}$ for any ${\bf p} \in \RR^2$ and reduced sublattice $\Lambda \subseteq \ZZ^2$, using the function $\textmono{nextpt}$.

\begin{corollary}
    \label{corollary:exists_hyperbola_convex_hull_alg}
    Given a standard basis $({\bf b}_1, {\bf b}_2)$ of a reduced lattice $\Lambda \subseteq \ZZ^2$, an $n >0$ and a ${\bf p} \in \RR^2$,
    Algorithm~\ref{alg:hyperbola_convex_hull} computes $\conv{H_{n,{\bf p}+\Lambda}}$ with time
    complexity in
    \[
        O\big( N\log(n+\det \Lambda)\big),
    \]
    where $N=\cardinal{\vertices{\conv{H_{n,{\bf p}+\Lambda}}}}$.
\end{corollary}

\begin{algorithm}
    \caption{Compute $\conv{H_{n,{\bf p}+\Lambda}}$.}
    \label{alg:hyperbola_convex_hull}

    \DontPrintSemicolon
    \SetKwInOut{Input}{Input}
    \SetKwInOut{Output}{Output}

    \Input{$n >0, {\bf p} \in \RR^2, \Lambda = \left\langle{\bf w}_1, {\bf w}_2\right\rangle \subseteq \QQ^2$, with~${\bf w}_1,~{\bf w}_2$~linearly~independent}
    \Output{$\textmono{hull} = \vertices{\conv{H_{n,{\bf p}+\Lambda}}}$}

    \SetKw{And}{and}

    \SetKwProg{Fn}{Function}{}{end}
    \Fn{\textmono{convex\_hull\_H\_n(}${\bf w}_1, {\bf w}_2\typeHint{\in \NonZero\Lambda}, n\typeHint{ >0}, {\bf p}\typeHint{\in \RR^2}$\textmono{)}} {
        ${\bf b}_1, {\bf b}_2\typeHint{\in \NonZero\Lambda} \gets \textmono{standard\_basis(}{\bf w}_1, {\bf w}_2\textmono{)}$\;
        ${\bf q}\typeHint{\in {\bf p}+\Lambda} \gets {\bf p} - \ceil{\frac{{\bf p}_x - 1}{{{\bf b}_1}_x}} {\bf b}_1$\;
        ${\bf q} \gets {\bf q} + \ceil{\frac{n/{\bf q}_x - {\bf q}_y}{{{\bf b}_2}_y}} {\bf b}_2$\;
        $\textmono{hull} \typeHint{\subseteq {\bf p}+\Lambda} \gets \{{\bf q}\}$\;

        \While {${\bf q}_y \ge 2$ \And $\textmono{nextpt(}{\bf b}_1,\,{\bf b}_2,\,n,\,{\bf q}\textmono{)} \ne \infty$} {
            ${\bf q} \gets \textmono{nextpt(}{\bf b}_1,\,{\bf b}_2,\,n,\,{\bf q}\textmono{)}$\;\label{line:nextpt}
            $\textmono{hull} \gets \textmono{hull} \cup \left\{{\bf q}\right\}$\;
            \label{line:union}
        }
        \If {${\bf q}_y \in (1,2)$} {
            ${\bf q} \gets ({\bf q}_y, {\bf q}_x)$\;
            ${\bf b}_1', {\bf b}_2'\typeHint{\in \NonZero\Lambda} \gets \textmono{standard\_basis(}({{\bf w}_1}_y, {{\bf w}_1}_x),({{\bf w}_2}_y, {{\bf w}_2}_x)\textmono{)}$\;
            ${\bf q}'\typeHint{\in {\bf p}+\Lambda} \gets {\bf q} - {{\bf b}_1'}$\;
            ${\bf q}' \gets {\bf q}' + \ceil{\frac{n/{\bf q}'_x - {\bf q}'_y}{{{\bf b}_2'}_y}} {{\bf b}_2'}$\;
            ${\bf q} \gets ({\bf q}'_y, {\bf q}'_x)$\;
            $\textmono{hull} \gets \textmono{hull} \cup \left\{{\bf q}\right\}$\;
            \label{line:union_2}
        }
        \Return $\textmono{hull}$\;
    }
\end{algorithm}

\begin{proof}
    Algorithm~\ref{alg:hyperbola_convex_hull} first computes the leftmost vertex of $\conv{H_{n,{\bf p}+\Lambda}}$
    as described in Proposition~\ref{prop:first_vertex}, and inserts it in $\textmono{hull}$. This part has complexity $O(1)$,
    since we are given a standard basis.
    Then it iterates $\textmono{nextpt}({\bf b}_1,{\bf b}_2, n, {\bf q})$ which, by Corollary~\ref{cor:nextpt_is_next_vertex} and Theorem~\ref{thm:exists_find_nextpt}, returns the subsequent vertices each in time $O(\log(n+\det \Lambda))$.
    This goes on until the rightmost vertex, for which $\Nextpt$ equals $\infty$,  is reached, with one exception: if we arrive at a vertex ${\bf q}$ with ${\bf q}_y\in (1,2)$ then we apply Lemma~\ref{lemma:exception}: we directly compute the last vertex of $\conv{H_{n,{\bf p} + \Lambda}}$ exactly as the first vertex was computed, but swapping the coordinates $x$ and $y$.

    Although $\textmono{hull}$ in the algorithm is formally a set the elements
    added to it are always new, so the union is really an insertion, taking constant time.
\end{proof}

\subsubsection{Starting at an intermediate point}

If we are given a value $x_{start}\in \RR_{>0}$, we can easily find the first vertex ${\bf q}$ of $\conv{H_{n,{\bf p} + \Lambda} \cap \{x\ge x_{start}\}}$ in $O(1)$ time, by a slight modification of the proof of Proposition~\ref{prop:first_vertex}.
However, such ${\bf q}$ may not be a vertex of $\conv{H_{n,{\bf p} + \Lambda}}$.
We now show that the first vertex of $\conv{H_{n,{\bf p} + \Lambda}}$ lying in $\{x\ge x_{start}\}$ can be found with a logarithmic number of iterations of  the same function $\textmono{nextpt}$ used in Algorithm~\ref{alg:hyperbola_convex_hull}.

\begin{proposition}
    \label{prop:find_next_vertex_from_non_vertex}
    Let ${\bf p} \in H_n$ be a point in the interior of $\conv{H_{n,{\bf p}+\Lambda}}$.
    Let $e_{\bf p}$ be the edge of $\conv{H_{n,{\bf p}+\Lambda}}$ vertically below
    ${\bf p}$ and let ${\bf v}$ be the right end-point of $e_{\bf p}$ (we admit ${\bf v}=\infty$).

    If ${\bf q} \coloneqq \Nextpt({\bf p}) \ne \infty$ then
    $\operatorname{dist}({\bf q}, e_{\bf p}) < \tfrac12 \operatorname{dist}({\bf p}, e_{\bf p})$
    and ${\bf q}_x \le {\bf v}_x$.
\end{proposition}

\begin{proof}
    By definition of ${\bf q} \coloneqq \Nextpt({\bf p})$, the slope of ${\bf q} - {\bf p}$ is smaller or equal than the slope of $e_{\bf p}$, that is, the slope of ${\bf v} - {\bf p}$. (This holds even if ${\bf v}=\infty$, in which case we consider the slope of ${\bf v} - {\bf p}$ to be zero).
    Hence, the ray ${\bf p} + \RR_{\ge 0}({\bf q} - {\bf p})$ intersects $e_{\bf p}$;  it must do so after passing through ${\bf q}$ because ${\bf q} \in \conv{H_{n,{\bf p}+\Lambda}}$.
    Since ${\bf v}$ is the right end-point of $e_{\bf p}$, we have ${\bf p}_x \le {\bf q}_x \le {\bf v}_x$.

    Now assume, to seek a contradiction, that $\operatorname{dist}({\bf q}, e_{\bf p}) \ge \tfrac12 \operatorname{dist}({\bf p}, e_{\bf p})$.
    This implies that
    ${\bf q}' \coloneqq {\bf p} + 2({\bf q} - {\bf p})$ is also on or above $e_{\bf p}$ and thus is in
    $H_{n,{\bf p}+\Lambda}$, further from ${\bf p}$ in the same direction as ${\bf q}$, in contradiction
    with the definition of ${\bf q}$.
    This situation is illustrated in Figure~\ref{fig:prop:nextpt-from-non-vertex}.
\end{proof}

\begin{figure}[htb]
    \includegraphics[width=.7\textwidth]{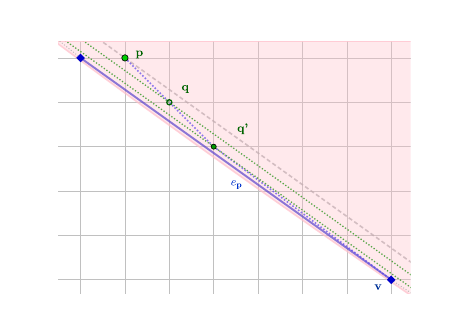}
    \caption{Case where $\nextpt_{\ZZ^2,H_n}$ is not the next vertex.}
    \label{fig:prop:nextpt-from-non-vertex}
\end{figure}

\begin{corollary}
    \label{cor:find_next_vertex_from_non_vertex_log}
Let $\Lambda\subset \ZZ^2$ be a reduced sublattice, let $n > 0$, and let ${\bf p} \in H_n$  a point with $\Nextpt({\bf p}) < \infty$.
    Let ${\bf v}$ be the leftmost vertex of
    \[
        \vertices{H_{n,{\bf p}+\Lambda}} \cap \{x > {\bf p}_x\},
    \]
    (possibly ${\bf v} = \infty$).
    Then, ${\bf v}$ is reached in at most $\log_2(n +\det \Lambda)$ iterative calls to $\textmono{nextpt}$ starting at ${\bf p}$.
\end{corollary}

\begin{proof}
    Let $e$ be the edge vertically
    below ${\bf p}$, so that ${\bf v}$ is its rightmost endpoint.
    If ${\bf p} - {\bf b}_2 \in H_n$ then ${\bf q} \coloneqq \Nextpt({\bf p})$ is in the same vertical line as ${\bf p}$ and has
    ${\bf q} - {\bf b}_2 \notin H_n$. Hence, we can assume without loss of generality that ${\bf p} - {\bf b}_2 \notin H_n$. That is, the vertical distance between ${\bf p}$ and $e$ is bounded by ${{\bf b}_2}_y \le \det(\Lambda)$.

    Let $x_e$ denote the $x$-coordinate of a primitive vector of $\Lambda$ parallel to $e$. Since $\Lambda$ is primitive, we have that $x_e =\det \Lambda$ if ${\bf v}=\infty$ and $0 < x_e \le {\bf v}_x < n+\det \Lambda$  if ${\bf v}$ is finite. The  vertical Euclidean distance between consecutive lattice lines parallel to $e$ equals $\det \Lambda/x_e$. As the distance between our point and $e$ is initially bounded by $\det \Lambda$ and is at least halved in each iteration, it will become smaller than the minimum possible vertical distance in
    \[
        \log_2 \frac{\det \Lambda}{\det \Lambda/x_e} = \log_2 (x_e) \le  \log_2(n+\det \Lambda)
    \]
    iterations. When this happens we must have reached a point ${\bf p}^* \in e$ and, if ${\bf p}^*\ne {\bf v}$, the next call to $\textmono{nextpt}$ finds ${\bf v}$. 
\end{proof}

Hence:

\begin{theorem}
\label{thm:next_vertex}
    Given a standard basis $({\bf b}_1, {\bf b}_2)$ of $\Lambda$ and an $x_{start} >0$, the vertex of $\vertices{\conv{H_n,{\bf p} + \Lambda}}\cap\{x\ge x_{start}\}$ with the smallest $x$-coordinate can be computed in $O(\log^2(n+\det \Lambda))$ arithmetic operations.
\end{theorem}

\begin{proof}
From the standard basis we compute a point ${\bf q}$ in the first vertical lattice line with $x$-coordinate at least $x_{start}$ as
    \[
    {\bf q} \coloneqq {\bf p} - \floor{\frac{{\bf p}_x - x_{start}}{{{\bf b}_1}_x}} {\bf b}_1.
    \]
    Then, the bottommost point in $H_{n, {\bf p}+\Lambda}$ along that line then equals
    \[
        {\bf q}  + \ceil{\frac{n/{\bf q}_x- {\bf q}_y }{{{\bf b}_2}_y}} {\bf b}_2.
    \]

    Hence, for the rest of the proof we assume that ${\bf p}$ is in fact the bottommost lattice point in the first vertical line with $x$-coordinate at least $x_{start}$. From there, Corollary~\ref{cor:find_next_vertex_from_non_vertex_log}
    says that in at most $\log_2 (n+\det \Lambda)$ calls to $\textmono{nextpt}$ we can find the next vertex after $x_{start}$.

    One problem is that we do not know, a priori, when to stop. But by Corollary~\ref{cor:nextpt_is_next_vertex} we know that all the iterations after we get the next vertex will give the subsequent vertices, in order. To really compute the first vertex we proceed as follows: we iterate  $\textmono{nextpt}$ $\log_2 (n+\det \Lambda)$ times and keep the list of points so obtained. The last one (plus perhaps some before it) are vertices. We then can backtrack, applying $\textmono{nextpt}$ in reverse (that is, reflecting everything along the diagonal $\{x=y\}$). Since we now start at a vertex, Corollary~\ref{cor:nextpt_is_next_vertex} again tells us that now all the points obtained will be vertices. The last one that coincides with what we computed when going forward is the next vertex after $x_{start}$ that we are looking for.
\end{proof}

\subsection{Proof of Theorem~\ref{thm:exists_find_nextpt}}
\label{subsec:rsmin}

\subsubsection{Casting rays at the hyperbola}
\label{subsubsec:cast}

One primitive that our algorithm repeatedly uses is to
find the farthest lattice point along a ray before crossing $h_n$. Here we show that this takes constant time.

\begin{lemma}
    \label{lemma:quadratic_formula}
    Let $a, b, c \in \RR$.
    There exists an algorithm which performs only $O(1)$ basic integer arithmetic (including
    at most one integer square root) to decide whether the quadratic equation
    $ax^2 + bx + c$ has zero, one, or two real roots, and to compute their integer floors.
\end{lemma}

\begin{proof}
    If $a = 0$, there is one root and its floor is $\lfloor -c/b\rfloor$.
    Otherwise, the exact roots are:
    \begin{equation*}
        r_1 = \frac{-b + \sqrt{b^2 - 4ac}}{2a}, \quad\quad r_2 = \frac{-b - \sqrt{b^2 - 4ac}}{2a}.
    \end{equation*}
    Since
    $\floor{\frac{p}{q}} = \floor{\frac{\lfloor p \rfloor}{q}}$
    whenever $q \in \ZZ_{>0}$, the floored roots can be rewritten as
    \begin{equation*}
        \floor{r_1} = \floor{\frac{-\sign{a}b + \floor{\sign{a}\sqrt{b^2 - 4ac}}}{2|a|}},
    \end{equation*}
    \begin{equation*}
        \floor{r_2} = \floor{\frac{-\sign{a}b - \ceil{\sign{a}\sqrt{b^2 - 4ac}}}{2|a|}}.
\qedhere
    \end{equation*}
\end{proof}

\begin{lemma}[Lattice ray cast at hyperbola]
    \label{lemma:exists_lattice_ray_cast_alg}
    Let $n > 0$, ${\bf p} \in \RR^2$ and ${\bf v} \in \RR^2 \setminus \left\{(0, 0)\right\}$.

    There exists an algorithm which performs only $O(1)$ basic integer arithmetic operations (including
    at most two integer square roots) to compute
    \begin{equation*}
        \raycast(H_n, {\bf p}, {\bf v}) \coloneqq
        \begin{cases}
            \min M & \text{ if } M \ne \emptyset, \\
                 \infty & \text{ otherwise,}
        \end{cases}
    \end{equation*}
    where
    \begin{equation*}
        M = \left\{m \in \ZZ_{\ge 0}: \cardinal{\left\{{\bf p} + m{\bf v}, {\bf p} + (m+1){\bf v}\right\} \cap H_n} = 1\right\}.
    \end{equation*}
\end{lemma}

\begin{proof}
    Consider the function $f: \RR \to \RR$ defined as
    \[
         f(m) \coloneqq ({\bf p}_x + m {\bf v}_x)({\bf p}_y + m {\bf v}_y) - n.
    \]
    Observe that  $f(m) = am^2 + bm + c$ with
    \[
        a = {\bf v}_x {\bf v}_y, \quad b = ({\bf v}_x {\bf p}_y + {\bf v}_y {\bf p}_x), \quad c = {\bf p}_x {\bf p}_y - n,
    \]

    Every number $m$ in the set $M$ satisfies that the segment $\{{\bf p} + (m+t){\bf v}: t \in [0, 1]\}$ crosses the
    hyperbola $\{xy = n\}$. Put differently, $M$ is contained in $\{r_1,r_1+1,r_2,r_2+1\}$ where $r_1$ and $r_2$ are the floors of the roots of $f$, computed by Lemma~\ref{lemma:quadratic_formula}. Hence, we only need to check which of these four (or less, if $f$ has less than two real roots, or if $r_2\in \{r_1,r_1+1\}$) points indeed lie in $M$, and compute the minimum of them.

    By definition, an $m\in \ZZ$ lies in $M$ if and only if $m\ge 0$ and the two points $\{{\bf p}+m{\bf v}, {\bf p}+(m+1){\bf v}\}$ satisfy that one of them has both coordinates positive and with product $\ge n$ while the other either has product of coordinates $<n$ or has both coordinates negative. This can be checked in constant time.
\end{proof}

\subsubsection{Right search bases}\label{subsubsec:rsb}

The following geometric property of two-dimen\-sio\-nal lattices will be used repeatedly in what follows:

\begin{lemma}
\label{lemma:convex_line}
    Let $X \subset \RR^2$ be a convex set, $\Lambda \subset \RR^2$ a lattice, and $P$ a half-plane delimited by a lattice line $\ell$.
    Let $\ell'$ be the next lattice line parallel to $\ell$ in the direction opposite to $P$.
    If $X \cap \ell' = \emptyset$ and $\cardinal{X \cap \ell} \ge 2$ then $X \cap \Lambda \subset P$.
\end{lemma}

\begin{proof}
    Let ${\bf p}, {\bf q}$ be two of the points in $X \cap \ell$, and suppose a lattice point
    ${\bf r}$ exists in $X \cap \Lambda \setminus P$.
    As all three points belong in $X$, which is convex, the triangle ${\bf p}{\bf q}{\bf r} \subseteq X$.
    Denote by ${\bf p}' = \overline{{\bf p}{\bf r}} \cap \ell'$, ${\bf q}' = \overline{{\bf q}{\bf r}} \cap \ell'$, and let
    $d = \operatorname{dist}({\bf r},\ell) / \operatorname{dist}(\ell, \ell') \in \ZZ_{\ge 2}$, the
    lattice distance between ${\bf r}$ and $\ell$.

    With this notation we have that ${\bf p}',{\bf q}' \in \tfrac1d\Lambda$ and, by Thales's Theorem,
    the length of segment $\ell' \cap {\bf p}{\bf q}{\bf r}$ must be $\tfrac{d - 1}{d} \abs{\overline{{\bf p}{\bf q}}}$.
    Hence, $\ell' \cap {\bf p}{\bf q}{\bf r}$ contains at least one lattice point in $X \cap \Lambda \cap \ell' = \emptyset$,
    so ${\bf r}$ cannot exist.
\end{proof}

Our algorithm to compute $\Nextpt({\bf p})$  is based on maintaining a basis of ${\bf p}+\Lambda$ with special properties.

\begin{definition}
    \label{def:right_search_basis}
    Let ${\bf p} \in H_n$ with ${\bf p} - {\bf b}_2 \notin H_n$.
    A \emph{right search basis for ${\bf p}$ in $H_{n,{\bf p}+\Lambda}$} is a lattice basis $(\textmono{in}, \textmono{out})$ of $\Lambda$
    satisfying
    \begin{enumerate}
        \item $\textmono{in}_x>0$, $\textmono{out}_x \ge 0$ and $\det(\textmono{in}, \textmono{out}) < 0$,
        \item ${\bf p} + \textmono{in} \in H_n$, ${\bf p} + \textmono{out} \notin H_n$.
        \item $\Nextpt({\bf p})  - {\bf p}$ lies in the closed cone generated by $\left\{\textmono{in}, \textmono{out}\right\}$.
        Equivalently, the slopes of $\textmono{in}$ and $\textmono{out}$ are, respectively, an upper and a lower bound for the slope of $\Minsl({\bf p})$.
    \end{enumerate}
\end{definition}

It is obvious that:

\begin{proposition}
    \label{prop:nextpt_in_cone}
    \label{prop:standard_is_rsb}
    If ${\bf b}_1, {\bf b}_2$ is the standard basis of $\Lambda$ then $({\bf b}_1, -{\bf b}_2)$ is a right search basis for every ${\bf p}\in H_n$ with ${\bf p}- {\bf b}_2 \notin H_n$.
\end{proposition}

\begin{proof}
    Condition (i) holds by construction.
    Condition (ii) follows for $\textmono{in} = {\bf b}_1$ from ${\bf p}\in H_n$ and ${{\bf b}_1}_x, {{\bf b}_1}_y \ge 0$, and it follows for $\textmono{out} = -{\bf b}_2$ by the hypothesis ${\bf p} - {\bf b}_2 \notin H_n$.
    For condition (iii) observe that the slope of $\Minsl({\bf p})$ lies in $[-\infty,0]$ by definition, and the slopes of ${\bf b}_1$ and $-{\bf b}_2$ are $\ge 0$ and $=-\infty$, respectively.
\end{proof}

Our algorithm to find $\Minsl({\bf p})$ and $\Nextpt({\bf p})$ (Algorithm~\ref{alg:compute_nextpt} below) starts with the standard basis and keeps updating a right search basis $(\textmono{in}, \textmono{out})$ relying for this in the following two lemmas, illustrated in Figure~\ref{fig:alg-update-steps}.

In both lemmas we have the following setting:
Let ${\bf p} \in H_n$ with ${\bf p} - {\bf b}_2 \notin H_n$ and let $(\textmono{in}, \textmono{out})$ be a right search basis for ${\bf p}$ in $H_{n,{\bf p}+\Lambda}$.
Since $ \Minsl({\bf p}) \in \cone{\textmono{in}, \textmono{out}}\cap \Lambda$ and $(\textmono{in}, \textmono{out})$ is a basis for $\Lambda$, there are unique $a,b\in \ZZ_{\ge 0}$ such that
\[
    \Minsl({\bf p}) = a\  \textmono{in} + b\ \textmono{out}.
\]

\begin{figure}[htb]
    \centering
    \begin{minipage}{.45\textwidth}
        \centering
        \includegraphics[width=.95\textwidth]{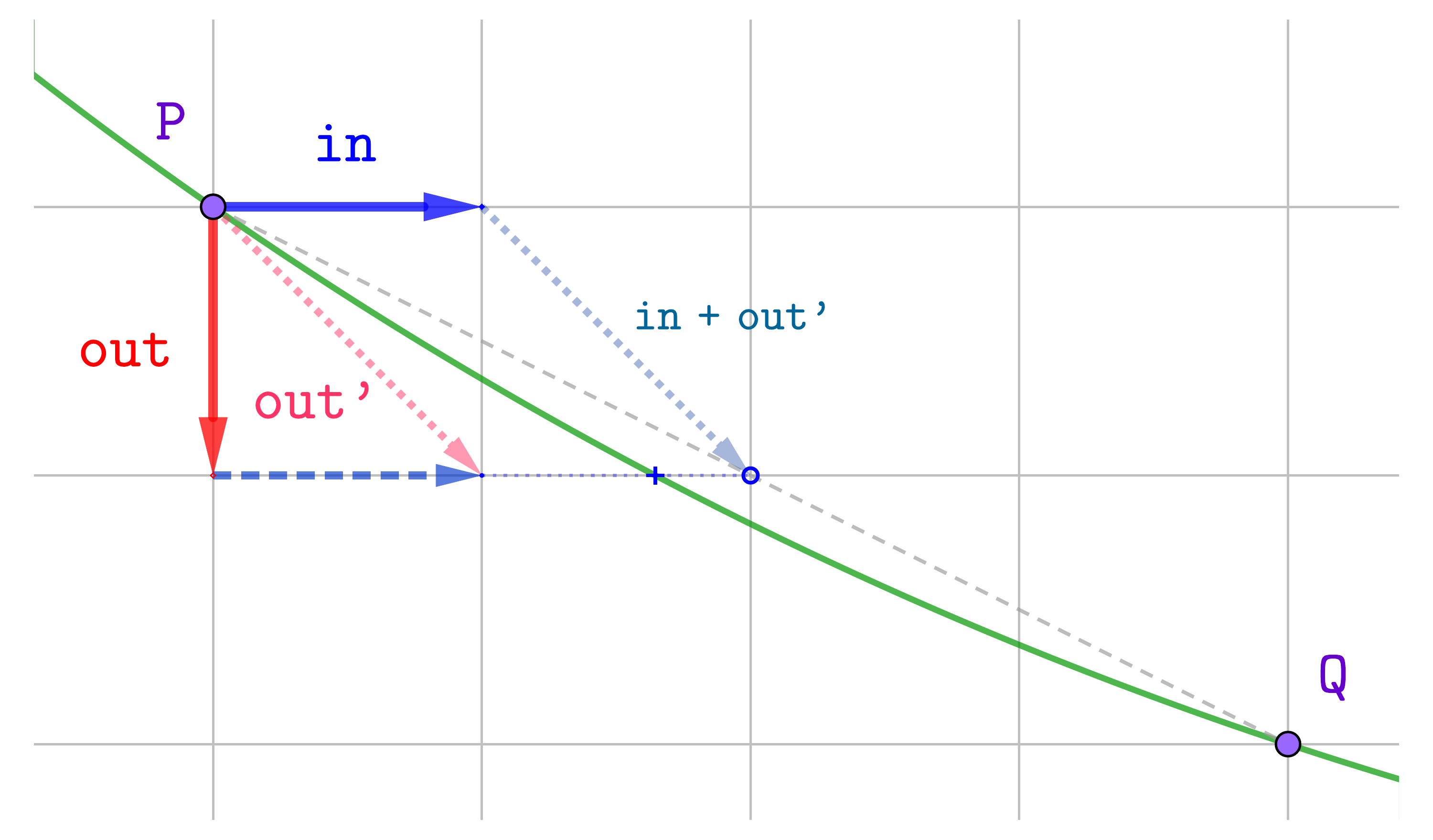}
    \end{minipage}%
    \hfill%
    \begin{minipage}{.45\textwidth}
        \centering
        \includegraphics[width=.95\textwidth]{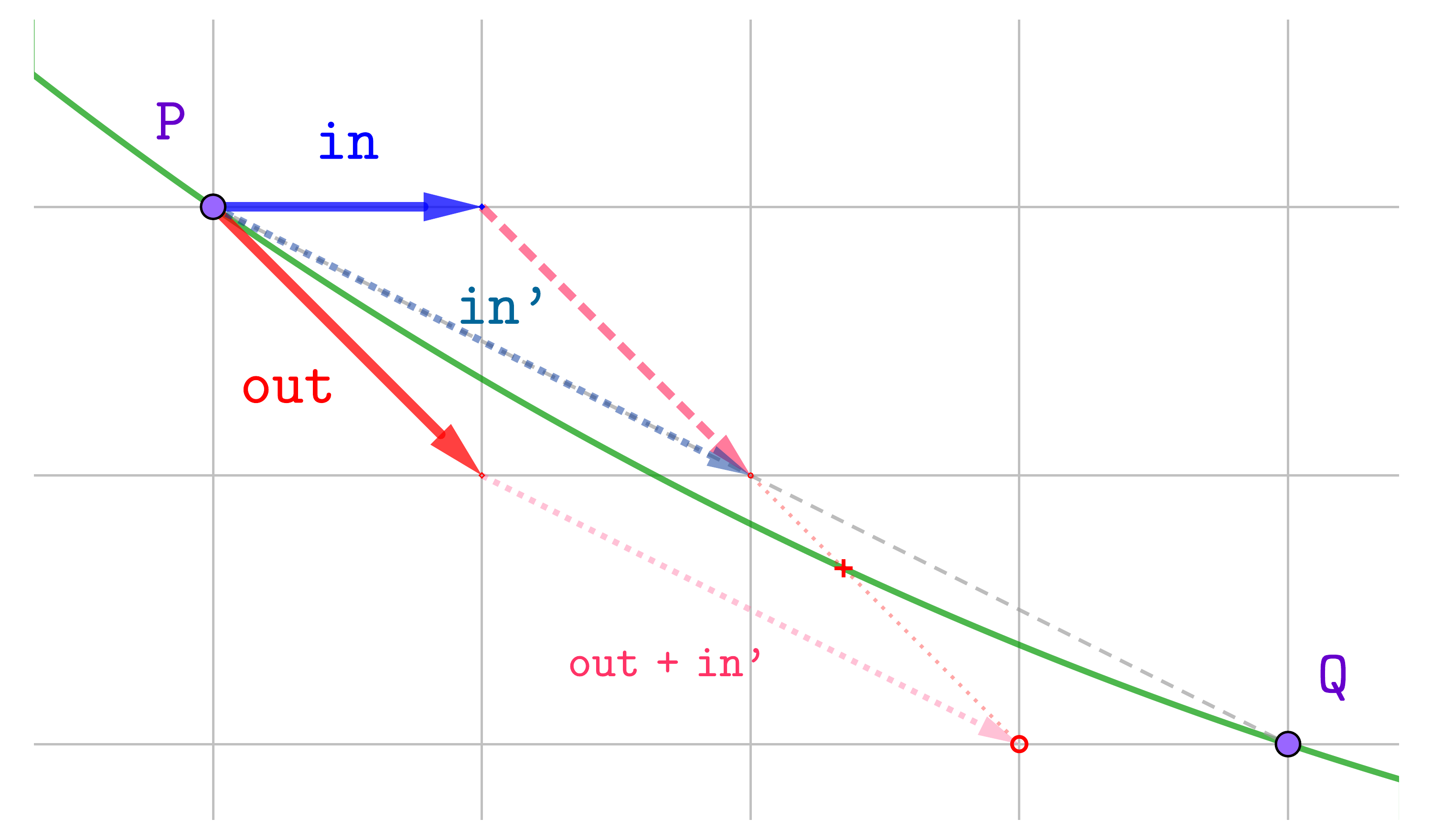}
    \end{minipage}
    \caption{Successive updates of the $\textmono{out}$ and $\textmono{in}$ bounds for $\Lambda = \ZZ^2$}
    \label{fig:alg-update-steps}
\end{figure}

\begin{lemma}
    \label{lemma:o}
     Let $\textmono{o} = \raycast(H_n, {\bf p} + \textmono{in}, \textmono{out})$ and $\textmono{in}' = \textmono{in} + \textmono{o} \cdot \textmono{out}$. Then:
    \begin{enumerate}
        \item $\textmono{o}\ne \infty$.
        \item $\Minsl({\bf p})$ lies in $\cone{\textmono{in}', \textmono{out}}$ but not in $\cone{\textmono{in}'+\textmono{out}, \textmono{out}}$.
        \item $(\textmono{in}', \textmono{out})$ is still a right search basis for ${\bf p}$.
        \item $\Minsl({\bf p}) = a' \textmono{in}' + b\ \textmono{out}$, where $a' = a \bmod b$.
    \end{enumerate}
\end{lemma}

\begin{proof}
        $\textmono{o}$ is chosen as the maximum non-negative integer such that
        ${\bf p} + \textmono{in} + \textmono{o}\cdot\textmono{out} = {\bf p} + \textmono{in}'$ is in $H_n$ while
        ${\bf p} + \textmono{in} + (\textmono{o} + 1)\cdot\textmono{out} = {\bf p} + \textmono{in}' + \textmono{out}$ is not.
        If this maximum was $\infty$, convexity of $H_n$ and the fact that ${\bf p} + \textmono{in}\in H_n$ implies that the whole ray ${\bf p} + \textmono{in}+ \RR_{\ge 0} \cdot \textmono{out}$ is contained in $H_n$.
        Since $H_n$ is convex and closed, this, in turn implies that any ray in the direction of $\textmono{out}$ and starting at a point of $H_n$ is also contained in $H_n$.
        In particular, this would happen for the ray starting at ${\bf p}$, but this contradicts the fact that ${\bf p} + \textmono{out} \notin H_n$. Hence, $\textmono{o}$ is never $\infty$.

        Parts (3) and (4) follow easily from (2), so we only need to prove (2).
        The first part of (2) follows from the fact that ${\bf p} + \textmono{in}'$ is in $H_n$, so that $\textmono{in}'$ is still an upper bpund for the slope of $\Minsl({\bf p})$. For the second part we apply Lemma~\ref{lemma:convex_line} as follows: Let $\ell$ be the line containing $\bf p$ and ${\bf p} +\textmono{in}'$ and let $\ell'$ be the line parallel to it passing through ${\bf p} +\textmono{out}$. Let
        \[
            X= ({\bf p} + [0,1)\cdot \textmono{in}' + \cone{\textmono{in}'+\textmono{out}, \textmono{out}}) \cup \{{\bf p} + \textmono{in}'\}.
        \]
        $X$ is convex since it is the union of a sum of convex sets (which is convex) and a point in the boundary of that sum. It contains the points $\bf p$ and ${\bf p} +\textmono{in}'$ from $\ell$, but no lattice point from $\ell'$ (since the only lattice points in $\ell' \cap X$ are $\textmono{out}$ and $\textmono{in}'+\textmono{out}$). Hence, no lattice point other than ${\bf p}$ lies in ${\bf p} + \cone{\textmono{in}'+\textmono{out}, \textmono{out}}$, which finishes the proof of (2).
\end{proof}

\begin{lemma}
    \label{lemma:i}
    Let $\textmono{i} = \raycast(H_n, {\bf p} + \textmono{out}, \textmono{in})$ and $\textmono{out}' = \textmono{out} + \textmono{i} \cdot \textmono{in}$. Then:
    \begin{enumerate}
        \item $\textmono{i} = \infty$ if and only if $\textmono{in} = \Minsl({\bf p})$. If this is not the case:
        \item $\Minsl({\bf p})$ lies in $\cone{\textmono{in}, \textmono{out}'}$. It only lies in $\cone{\textmono{in}, \textmono{in}+\textmono{out}'}$ if $\textmono{in}+\textmono{out}' = \Minsl({\bf p})$.
        \item $(\textmono{in}, \textmono{out}')$ is still a right search basis for ${\bf p}$.
        \item $\Minsl({\bf p}) = a'\ \textmono{in} + b\ \textmono{out}'$, where $a' = [(a-1) \bmod b] + 1$. (That is, $a' = a \bmod b$ unless $a$ divides $b$, in which case $a'=b$).
    \end{enumerate}
\end{lemma}

\begin{proof}
        One direction of part (1) is obvious: if $\textmono{in}=\Minsl({\bf p})$ then no  point of $H_{n,\Lambda}\cap\{x \ge {\bf p}_x\}$ lies below the line starting at ${\bf p}$ and with direction $\textmono{in}$.
        In particular,  no  point of $H_{n,\Lambda}$ lies in the ray starting at ${\bf p} +\textmono{out}$ with direction $\textmono{in}$, which implies that $\textmono{i} = \infty$.

        For the converse, let us consider the closed half-plane $P$ with ${\bf p}$ and ${\bf p} +\textmono{out}$ on its boundary and ${\bf p} +\textmono{in}$ in its interior.
        Let $X = P \cap H_n$ and let us apply Lemma~\ref{lemma:convex_line} to it, with $\ell$ being the line containing ${\bf p}$ and  ${\bf p} +\textmono{in}$, and $\ell'$ the one containing ${\bf p} +\textmono{out}$ and  ${\bf p} +\textmono{in}+\textmono{out}$.
        The fact that $\textmono{in}=\infty$ implies that $\ell'$ contains no lattice point of $X$, so the Lemma says that no point of $X$ lies below the line $\ell$.
        Hence, $\textmono{in}=\Minsl({\bf p})$.

        Parts (3) and (4) again follow easily from (2), taking into account the exceptional case $\textmono{in}+\textmono{out}' = \Minsl({\bf p})$ in part (2). In this case, $a=\textmono{i}+1$, $b=1$ and $a'=b=1$, as claimed. Conversely, if $a$ divides $b$, the fact that $\Minsl({\bf p})$ is primitive implies that $a=1$. Hence,
    \[
    \Minsl({\bf p}) =   \textmono{in} + b\ \textmono{out}.
    \]
    The definition of $\Minsl({\bf p})$ then implies that $b$ is the smallest integer for which $\textmono{in} + b\ \textmono{out}$ is in $H_n$ or, equivalently, that $b= \textmono{i}+1$.

        So we only need to prove (2).
        The fact that $\Minsl({\bf p})$ only lies in $\cone{\textmono{in}, \textmono{in}+\textmono{out}'}$ if $\textmono{in}+\textmono{out}' = \Minsl({\bf p})$ is easy: $\Minsl({\bf p})$ certainly lies in $\cone{\textmono{in}+\textmono{out}', \textmono{out}}$ because $(textmono{in}, \textmono{out})$ was a right search basis and ${\bf p} +\textmono{in}+\textmono{out}' \in H_n$.
        To show that $\Minsl({\bf p})$ lies in $\cone{\textmono{in}, \textmono{out}'}$ we are going to apply Lemma~\ref{lemma:convex_line} to the convex set $X = H_n \cap \{x>{\bf p}_x\}$.
    Observe that the four points ${\bf p}$, ${\bf p}+ \textmono{in}$, ${\bf p} + \textmono{out} + \textmono{i} \cdot \textmono{in}$ and ${\bf p} + \textmono{out} + (\textmono{i} + 1)\cdot \textmono{in}$ form a parallelepiped with
    \begin{align*}
        &{\bf p}\in H_n\setminus X, \qquad\qquad {\bf p} + \textmono{out} + \textmono{i} \cdot \textmono{in} = {\bf p} + \textmono{out}' \notin H_n,
        \\
        &{\bf p}+ \textmono{in}\in X, \quad\qquad {\bf p} + \textmono{out} + (\textmono{i} +1) \cdot \textmono{in} = {\bf p} + \textmono{out}' + \textmono{in} \in X.
    \end{align*}

    The parallelepiped is unimodular since
    \[
        \det(\textmono{out} + \textmono{i} \cdot \textmono{in}, \textmono{in}) = \det(\textmono{out}, \textmono{in}) = \det(\Lambda).
    \]
    Hence, the lines $\ell$ and $\ell'$ containing respectively $\{{\bf p}+ \textmono{in}, \quad {\bf p} + \textmono{out}' +  \textmono{in}\}$ and $\{{\bf p}, {\bf p} + \textmono{out}' \}$ are consecutive parallel lines, with $\ell$ above $\ell'$.
    The first one contains at least the two lattice points of $X$ defining it, and the second one contains no lattice point of $X$ because of the following argument.
    If it did, such lattice points must be of the form
    \[
        {\bf p} + k(\textmono{out} + \textmono{i} \cdot \textmono{in})
    \]
    for some $k>1$ and convexity of $H_n$ would then imply that ${\bf p} + \textmono{out} + \textmono{i} \cdot \textmono{in} \in H_n$, a contradiction with the definition of $\textmono{i}$.
    Hence, Lemma~\ref{lemma:convex_line} implies that all lattice points of $X$ lie on or above $\ell$, which is equivalent to saying that $\textmono{out}'$ is a lower bound for the slope of $\Minsl({\bf p})$, hence $\Minsl({\bf p}) \in \cone{\textmono{in}, \textmono{out}'}$.
\end{proof}

\subsubsection{The algorithm}\label{subsubsec:rsmin}

Theorem~\ref{thm:exists_find_nextpt} is proven by Algorithm~\ref{alg:compute_nextpt}, summarized as follows:
After computing a standard basis and moving the point ${\bf p}$ vertically down as much as possible, the algorithm uses the standard basis as an initial right search basis and then updates it by iterating the procedures of Lemmas~\ref{lemma:o} and~\ref{lemma:i} in turns.

\begin{algorithm}
    \caption{Compute $\Nextpt({\bf p})$.}
    \label{alg:compute_nextpt}

    \DontPrintSemicolon
    \SetKwInOut{Input}{Input}
    \SetKwInOut{Output}{Output}

    \Input{$\Lambda = \left\langle{\bf b}_1, {\bf b}_2\right\rangle \subseteq \QQ^2$ with
        ${{\bf b}_1}_x > 0, {{\bf b}_2}_x = 0, 0 \le {{\bf b}_1}_y < {{\bf b}_2}_y$, $n >0 $, ${\bf p} \in H_n$.}
    \Output{$\Nextpt({\bf p})$}

    \SetKwProg{Fn}{Function}{}{end}
    \Fn{\textmono{nextpt(}${\bf b}_1$, ${\bf b}_2\typeHint{\in \NonZero\Lambda}$, $n\typeHint{ > 0}$, ${\bf p}\typeHint{\in H_n}$\textmono{)}$\typeHint{\to {\scriptstyle H_{n,{\bf p}+\Lambda}\,\cup\;\left\{\infty\right\}}}$} {
        \If {${\bf p} - {\bf b}_2 \in H_n$} {\Return ${\bf p} - \textmono{raycast(}H_n,\,{\bf p},\,-{\bf b}_2\textmono{)}\cdot{\bf b}_2$}
        $\textmono{in},\,\textmono{out}\typeHint{\in \NonZero\Lambda} \gets {\bf b}_1,\,-{\bf b}_2$\;
        $\textmono{minsl} \typeHint{\in \NonZero\Lambda \cup \left\{\square\right\}} \gets \square$\;
        \While {$\textmono{minsl} = \square$} {
            $\textmono{o} \typeHint{\in \ZZ_{\ge 0}} \gets \textmono{raycast(}H_n,\,{\bf p} + \textmono{in},\,\textmono{out}\textmono{)}$\;
            \label{line:m_in}
            $\textmono{in} \gets \textmono{in} + \textmono{o}\cdot\textmono{out}$\;
            \label{line:update_in}
            $\textmono{i} \typeHint{\in \ZZ_{>0} \cup \left\{\infty\right\}} \gets \textmono{raycast(}H_n,\,{\bf p} + \textmono{out},\,\textmono{in}\textmono{)}$\;
            \label{line:m_out}
            \uIf {$\textmono{i} < \infty$} {
                $\textmono{out} \gets \textmono{out} + \textmono{i}\cdot\textmono{in}$
                \label{line:update_out}
            } \lElse {$\textmono{minsl} \gets \textmono{in}$}
        }
        $\textmono{m} \typeHint{\in \ZZ_{>0} \cup \left\{\infty\right\}} \gets \textmono{raycast(}H_n,\,{\bf p},\,\textmono{minsl}\textmono{)}$\;
        \uIf {$\textmono{m} = \infty$} {\Return $\infty$}
        \lElse {\Return ${\bf p} + \textmono{raycast(}H_n,\,{\bf p},\,\textmono{minsl}\textmono{)}\cdot\textmono{minsl}$}
    }
\end{algorithm}

The correctness and running time of the algorithm follows from the fact that at every iteration the search cone is narrowed, as illustrated in Figure~\ref{fig:alg-update-steps}, performing an agile binary search within the Stern-Brocot tree~\cite{Yan2013Number} until the slope of $\Minsl({\bf p})$. In other words, the algorithm is essentially following the same steps as the Euclidean algorithm would follow to compute a $\gcd$.

\begin{proof}[Proof of Theorem~\ref{thm:exists_find_next_vertex}]
    In the first lines of Algorithm~\ref{alg:compute_nextpt} we compute a standard basis and use it to output the bottommost lattice point of $H_n$ in the line $\{x={\bf p}_x\}$ if ${\bf p} \setminus {\bf b}_2 \in H_n$.
    If this is not the case then the algorithm uses the standard basis as an initial right search basis (Proposition~\ref{prop:standard_is_rsb}).
    After that the algorithm enters a loop in which the right search basis is modified as in Lemmas~\ref{lemma:o} and~\ref{lemma:i} (hence it remains a right search basis) and it only exits the loop when $\textmono{i}=\infty$ which, by Lemma~\ref{lemma:i}, implies that we have indeed found $\Minsl({\bf p}) = \textmono{in}$.
    The lines after the loop compute $\Nextpt({\bf p})$ from $\Minsl({\bf p})$.

    Each iteration of the ``while'' loop takes $O(1)$ operations by Lemma~\ref{lemma:exists_lattice_ray_cast_alg}. The only thing left to be proven is that, assuming $\Lambda$ reduced and ${\bf p}_y\notin (1,2)$, the algorithm terminates in $O(\log(n+\det \Lambda))$ iterations of the loop. To show this, let $a$ and $b$ be the coefficients from Lemmas~\ref{lemma:o} and~\ref{lemma:i} for our initial situation $\textmono{in}={\bf b}_1, \textmono{out}=-{\bf b}_2$. That is, $a, b\in \ZZ_{> 0}$ and
    \begin{align}
    \label{eq:a_and_b}
        \Minsl({\bf p}) = a\ {\bf b}_1 - b\ {\bf b}_2.
    \end{align}
    Consider the following algorithm to compute $\gcd(a,b)$:

    \begin{algorithm*}[H]
        \caption{Greatest common divisor.}
        \label{alg:gcd}

        \DontPrintSemicolon
%
%
        \SetKwProg{Fn}{Function}{}{end}
        \Fn{\textmono{gcd(}$\textmono{a},\textmono{b}\typeHint{\in \ZZ_{>0}}$\textmono{)}$\typeHint{\to {\scriptstyle \ZZ_{>0}}}$} {
            \While {$\textmono{a}  \bmod\textmono{b} \ne 0$} {
                $\textmono{a} \gets \textmono{a}  \bmod\textmono{b}$\;
                $\textmono{b} \gets ((\textmono{b}-1) \bmod \textmono{a}) + 1$\;
            }
            \Return $\textmono{b}$\;
        }
    \end{algorithm*}
    This is the usual Euclidean algorithm, except we break the symmetry between $a$ and $b$ and take the representatives for $\textmono{b} \bmod \textmono{a}$ in $\{1,\dots, \textmono{a}\}$ instead of $\{0, \dots, \textmono{a}-1\}$.
    Thus, when the standard algorithm would get $\textmono{b}=0$ (and return $\textmono{a}$) ours gets $\textmono{b} = \textmono{a}$, which makes it exit the loop and return $\textmono{b}$.
    (When the standard algorithm would make $\textmono{a}=0$ and return $\textmono{b}$, ours does the same).

    Now, part (4) of Lemmas~\ref{lemma:o} and~\ref{lemma:i} implies that if we perform simultaneously each iteration of Algorithms~\ref{alg:compute_nextpt} and~\ref{alg:gcd}, the equation
    \[
        \Minsl({\bf p}) = \textmono{a}\ \textmono{in} + \textmono{b}\ \textmono{out}
    \]
    remains always true.
    Hence, both algorithms do the same number of iterations and this number is well-known to lie in $O(\log(\min\{a,b\}))$ for Algorithm~\ref{alg:gcd}. Hence, this is what to need to bound.

    For this, the fact that $\Lambda$ is reduced implies that ${{\bf b}_1}_x =1$ and ${{\bf b}_2}_x =0$, so Eq.~\ref{eq:a_and_b} implies that $a$ equals the $x$ coordinate of $\Minsl({\bf p})$.
    By Lemma~\ref{lemma:exception}, this is bounded by $2n+2\Delta$.
\end{proof}

\begin{remark}
    \label{remark:generalized_algorithm_condition}
    The correctness of Algorithm~\ref{alg:compute_nextpt} follows from Lemmas~\ref{lemma:o} and~\ref{lemma:i}, which rely solely on the convexity of the hyperbola.
    Hence, if $H$ is the region above the graph of a convex function $f: \RR^2 \to \RR^2$, our algorithms can easily be adapted to compute $\conv{H \cap ({\bf p} + \Lambda)}$ as long as there is a lattice ray intersection algorithm $\raycast(H, {\bf p}, {\bf v})$ analogous to the one of Lemma~\ref{lemma:exists_lattice_ray_cast_alg}.

    Indeed, let ${\bf p} \in H$ be such that ${\bf p} - {\bf b}_2 \notin X$, where ${\bf b}_1, {\bf b}_2$ is the standard basis of $\Lambda$.
    The standard basis ${\bf b}_1, {\bf b}_2$ may no longer be a right search basis, but one right search basis for  ${\bf p}$ can be obtained using $\textmono{out}=-{\bf b}_2$ and computing $\textmono{in}$ via a single lattice ray cast from ${\bf p} + {\bf b}_1$ in the direction of $\pm {\bf b}_2$, depending on whether ${\bf p} + {\bf b}_1$ is in $X$ or not. The rest of our algorithms work with no change (although their complexity will depend on the particular $f$).
\end{remark}

\printbibliography

\end{document}